\documentclass[12pt]{article}
\usepackage{amssymb}
\usepackage{amsmath}
\usepackage{t1enc}
\usepackage[latin1]{inputenc}
\usepackage[english]{babel}
\usepackage{amsmath,amsthm}
\usepackage{amsfonts}
\usepackage{latexsym}
\usepackage[dvips]{graphicx}
\usepackage{graphicx}
\usepackage[natural]{xcolor}

\newtheorem{theorem}{Theorem}
\newtheorem{lemma}[theorem]{Lemma}

\newtheorem{example}[theorem]{Example}

\newtheorem{remark}[theorem]{Remark}

\input{tcilatex}
\begin{document}

\date{27 June 2012}
\author{Ruimin Xu$^{1,2}\thanks{
Corresponding author.},$ Liangquan
Zhang$^{1,3}$ \\
1. School of Mathematics, Shandong University\\
Jinan 250100, People's Republic of China.\\
2. School of Mathematics, Shandong Polytechnic University, \\
Jinan, 250353, People's Republic of China.\\
E-mail: ruiminx@126.com.\\
{3. }Laboratoire de Math\'eatiques, \\
Universit\'e de Bretagne Occidentale, 29285 Brest C\'edex, France.\\
E-mail:xiaoquan51011@163.com }

\title{Stochastic Maximum Principle for Mean-field Controls and Non-Zero Sum
Mean-field Game Problems for Forward-Backward Systems}

\maketitle

\begin{abstract}
The objective of the present paper is to investigate the solution of fully
coupled mean-field forward-backward stochastic differential equations
(FBSDEs in short) and to study the stochastic control problems of mean-field
type as well as the mean-field stochastic game problems both in which state
processes are described as FBSDEs. By combining classical FBSDEs methods
introduced by Hu and Peng [Y. Hu, S. Peng, Solution of forward-backward
stochastic differential equations, Probab. Theory Relat. Fields 103 (1995)]
with specific arguments for fully coupled mean-field FBSDEs, we prove the
existence and uniqueness of the solution to this kind of fully coupled
mean-field FBSDEs under a certain \textquotedblleft monotonicity" condition.
Next, we are interested in optimal control problems for (fully coupled
respectively) FBSDEs of mean-field type with a convex control domain. Note
that the control problems are time inconsistent in the sense that the
Bellman optimality principle does not hold. The stochastic maximum principle
(SMP) in integral form for mean-field controls, which is different from the
classical one, is derived, specifying the necessary conditions for
optimality. Sufficient conditions for the optimality of a control is also
obtained under additional assumptions. Then we are concerned the maximum
principle for a new class of non-zero sum stochastic differential games.
This game system differs from the existing literature in the sense that the
game systems here are characterized by (fully coupled respectively) FBSDEs
in the mean-field framework. Our paper deduces necessary conditions as well
as sufficient conditions in the form of maximum principle for open
equilibrium point of this class of games respectively.
\end{abstract}


\section{Introduction}

In this paper, we consider the fully coupled forward-backward stochastic
differential equations (FBSDEs) of mean-field type
\begin{eqnarray}
X_{t} &=&x+\int_{0}^{t}\mathbb{E}^{\prime }[b(s,X_{s}^{\prime
},Y_{s}^{\prime },Z_{s}^{\prime },X_{s},Y_{s},Z_{s})]ds+\int_{0}^{t}\mathbb{E%
}^{\prime }[\sigma (s,X_{s}^{\prime },Y_{s}^{\prime },Z_{s}^{\prime
},X_{s},Y_{s},Z_{s})]dW_{s},  \notag \\
Y_{t} &=&\Phi (X_{T})+\int_{t}^{T}\mathbb{E}^{\prime }[f(s,X_{s}^{\prime
},Y_{s}^{\prime },Z_{s}^{\prime
},X_{s},Y_{s},Z_{s})]ds-\int_{t}^{T}Z_{s}dW_{s},  \label{fully coupled FBSDE}
\end{eqnarray}%
where $b,f:\Omega \times \Omega \times \lbrack 0,T]\times \mathbb{R}\times
\mathbb{R}\times \mathbb{R}^{d}\times \mathbb{R}\times \mathbb{R}\times
\mathbb{R}^{d}\longrightarrow \mathbb{R}$ and $\sigma :\Omega \times \Omega
\times \lbrack 0,T]\times \mathbb{R}\times \mathbb{R}\times \mathbb{R}%
^{d}\times \mathbb{R}\times \mathbb{R}\times \mathbb{R}^{d}\longrightarrow
\mathbb{R}^{d}$ satisfy the \textquotedblleft monotone\textquotedblright\
condition introduced firstly by Hu and Peng \cite{Hu Peng} and $W$ is a $d-$%
dimensional Brownian motion. Here, the coefficients $\mathbb{E}^{\prime
}[\phi (s,X_{s}^{\prime },Y_{s}^{\prime },Z_{s}^{\prime
},X_{s},Y_{s},Z_{s})]\ (\phi =b,\sigma ,f),$ which are different from the
classical coefficients of fully coupled FBSDEs, can be interpreted as
\begin{equation*}
\mathbb{E}^{\prime }[\phi (s,X_{s}^{\prime },Y_{s}^{\prime },Z_{s}^{\prime
},X_{s},Y_{s},Z_{s})]=\int_{\Omega }\phi (\omega ^{\prime },\omega
,s,X_{s}(\omega ^{\prime }),Y_{s}(\omega ^{\prime }),Z_{s}(\omega ^{\prime
}),X_{s}(\omega ),Y_{s}(\omega ),Z_{s}(\omega ))P(d\omega ^{\prime }).
\end{equation*}

(Fully coupled) FBSDEs are encountered in the probabilistic interpretation
(Feynman-Kac formula) of a large kind of second order quasi-linear PDEs,
mathematical economics, mathematical finance and especially in the
stochastic control problems (cf. \cite{Peng 1991}-\cite{Shi Wu 2010}). There
have been many results on the solvability of fully-coupled FBSDEs. Antonelli
\cite{Antonelli} first studied these equations, and he proved the existence
and uniqueness with the help of the fixed point theorem when the time
duration $T$ is sufficiently small. Among others, to our knowledge, there
exist three main methods to investigate the solvability of an FBSDEs on an
arbitrarily prescribed time duration. The first one concerns a kind of
\textquotedblleft four step scheme\textquotedblright\ by Ma et al. \cite%
{Four step} which can be regarded as a sort of combination of methods of
PDEs and probability. The second one is the purely probabilistic method by
Hu and Peng \cite{Hu Peng}, Peng and Wu \cite{Peng Wu}, Yong \cite{Yong} and
Pardoux and Tang \cite{Pardoux Tang}. They required the \textquotedblleft
monotonicity\textquotedblright\ condition on the coefficients. The third one
is motivated by the study of numerical methods for some linear FBSDEs (see
Delarue \cite{Delarue} and Zhang \cite{Zhang}). Delarue \cite{Delarue}
relied on PDEs arguments, so its coefficients have to be deterministic while
Zhang \cite{Zhang} imposed some assumptions on the derivatives of the
coefficients instead of the monotonicity condition.

Buckdahn, Djehiche, Li, and Peng \cite{MFBSDE1} and Buckdahn and Li et al.
\cite{MFBSDE2} investigated a new kind of BSDEs-Mean-field BSDEs, inspired
by Lasry and Lions \cite{Lions}. In the present work, we adapt the methods
developed by Hu and Peng \cite{Hu Peng} in order to establish the existence
and uniqueness result for the fully coupled mean-field FBSDEs under the
``monotone" condition. The two technical lemmas, aiming to prove the
existence result of fully coupled mean-field FBSDEs, differ from the
classical lemma in \cite{Hu Peng} because of the mean-field type. When the
coefficients $b, \sigma$ and $f$ do not depend on $\omega^{\prime }$, the
fully coupled equation (\ref{fully coupled FBSDE}) reduces to the standard
one. So our result is nontrivially more general of \cite{Hu Peng}.

We also consider stochastic optimal control problems and stochastic
differential games (SDGs) in which the state variables are described by a
system of mean-field FBSDEs. Mean-field control problems were recently
studied by many researchers, such as Andersson, Djehiche\cite{Djehiche},
Buckdahn, Djehiche and Li \cite{Juan Li General SMP}, Meyer-Brandis, $%
\emptyset$sendal, and Zhou \cite{Zhou} and Li \cite{Juan Li Automatica}.
Andersson, Djehiche\cite{Djehiche} use the methods in Bensoussan \cite%
{Bensoussan} to obtain the necessary conditions of the optimality of a
control, i.e. they suppose that the control state space is convex so as to
make a convex perturbation of the optimal control and obtain a maximum
principle of local condition. Buckdahn, Djehiche and Li \cite{Juan Li
General SMP} get a Peng's type maximum principle for a general action space
where the action space is not convex, using a spike variation of the optimal
control. In Meyer-Brandis, $\emptyset$sendal and Zhou \cite{Zhou}, a
stochastic maximum principle of mean-field type in a similar setting is
studied, but by using Malliavin calculus. Li \cite{Juan Li Automatica}, also
using the convex perturbation technology with the convex assumption for
control domain, has a different controlled system and state equation of
mean-field type from \cite{Djehiche}.

However, the results above are all on the forward control system. As far as
we know, Peng \cite{Peng 1993} originally studied one kind of
forward-backward stochastic control system which has the economic background
and could be used to study the recursive optimal control problem in the
mathematical finance. He obtained the maximum principle for this kind of
control system with the control domain being convex. Later, Shi and Wu \cite%
{Shi Wu 2006} applied the spike variational technique to derive the maximum
principle for fully coupled forward-backward stochastic control system in
the global form and indicated that the control domain is not necessarily
convex but the control variable can't enter into the diffusion term. In
order to study the forward-backward stochastic control problem under the
mean-field framework, we apply the convex perturbation methods introduced in
Bensoussan \cite{Bensoussan} and analytical technique provided by \cite{Juan
Li Automatica} to establish a necessary condition for optimality of the
control in the form of the maximum principle for the (fully coupled
respectively) mean-field forward-backward stochastic control system in which
the state equation is mean-field FBSDE (fully coupled mean-field FBSDE
respectively). The adjoint equation, playing an important role in deriving
the SMP, is a (fully coupled respectively) mean-field backward SDE and has a
unique adapted solution under the given assumptions with the help of the
conclusion in \cite{MFBSDE2} (or the conclusion in Theorem 3.1
respectively). Also, we obtain the corresponding sufficient condition, which
can check whether the candidate optimal control is optimal or not. Our
results can be reviewed as an extension of Peng \cite{Peng 1993} and Li \cite%
{Juan Li Automatica}.

Inspired by Wang and Yu \cite{Wang Yu}, which gave the maximum principle for
non-zero sum differential games of BSDE system, we study the non-zero sum
stochastic differential games (SDGs in short) of mean-field type.
Differential games, originally studied by Isaacs \cite{Isaacs 1965}, are
ones in which the position, being controlled by players, evolves
continuously. Fleming and Souganidis \cite{Fleming Souganidis 1989} were the
first to study in a rigorous manner two-player zero sum SDGs. Their work has
translated former results on differential games by Isaacs \cite{Isaacs 1965}%
, Friedman \cite{Friedman 1971}, and, in particular, Evans and Souganidis
\cite{Evans Souganidis 1984} from the purely deterministic into the
stochastic framework and has given an important impulse for the research in
the theory of stochastic differential games. Next, the advances in SDGs
appear over a large number of fields (cf. \cite{Hamadene 1999}-\cite{Altman
2005}).

We notice that the game literature is mainly restricted to forward
(stochastic) systems, i.e., these game systems are described by forward
(stochastic) differential equations. Recently, Wang and Yu \cite{Wang Yu}
concerned the theory of backward stochastic differential games and obtained
the maximum principle as well as the verification theorem for non-zero sum
SDGs of BSDEs in which game systems are described by BSDEs. It is remarkable
that this topic about the forward-backward system is quite lacking in
literature. To fill the gap, we investigate the theory of forward-backward
SDG problems under the mean-field framework. Similar to our stochastic
control problems, we study the SDGs with the state equation having two
different forms: mean-field FBSDEs and fully coupled mean-field FBSDE. By
virtue of an argument of the convex perturbation, we deduce the stochastic
maximum principle for the equilibrium point of Problem (FBNZ) (Problem
(CFBNZ) respectively), which gives the candidate equilibrium points. By
extending classical approaches to the mean-field framework, we prove, under
some restrictive assumptions (but comparable with those in the classical
case), the sufficiency of the necessary conditions. It is necessary to point
that our SDGs conclusion not only extends the result of Wang and Yu \cite%
{Wang Yu} but also includes the situation where the state equation of the
stochastic game system is classical (i.e. in no mean-field form) FBSDE
(fully coupled FBSDE respectively).

Our paper is organized as follows. Section 2 recalls some elements of the
theory of FBSDEs and mean-field BSDEs which are needed in what follows.
Section 3 investigates the uniqueness and existence of the solution of fully
coupled mean-field FBSDEs under the \textquotedblleft
monotonicity\textquotedblright\ condition in which two technical lemmas are
used to prove the existence result. In Section 4, we study the
forward-backward stochastic control system of mean-field type. Specifically,
the maximum principle, specifying the necessary condition for optimality, is
deduced and we get, under additional assumptions, the corresponding
sufficient condition which can check whether the candidate control is
optimal or not. Similar results about fully coupled forward-backward
stochastic control system of mean-field type are obtained in Section 5.
Following the idea introduced in Section 4 and Section 5, we analyze the
non-zero sum stochastic differential games of FBSDEs and fully coupled
FBSDEs in Section 6 and Section 7, respectively, and derive the necessary
condition in the form of the maximum principle as well as the sufficient
condition--verification theorem for the equilibrium point.

\section{Preliminaries}

Let $(\Omega,\mathcal{F},\mathcal{F}_t,P)$ be a given complete filtered
probability space on which a $d$-dimensional standard Brownian motion $%
W=(W_t)_{t\geq 0}$ is defined. By $\mathbb{F}=\{\mathcal{F}_t,{0 \leq t\leq T%
}\}$ we denote the natural filtration of $W$ augmented by $P-$null sets of $%
\mathcal{F}$, i.e.,
\begin{equation*}
\mathcal{F}_t= \sigma\{W_s, s\leq t\}\vee \mathcal{N}_P, \ \ t\in [0,T],
\end{equation*}
where $\mathcal{N}_P$ is the set of all $P$-null sets and $T>0$ is a fixed
time horizon.

We shall introduce the following two processes which can be used frequently
in what follows:
\begin{eqnarray*}
\mathcal{S}_{\mathbb{F}}^{2}(0,T;\mathbb{R}) &:&=\bigg\{(\phi _{t})_{0\leq
t\leq T}\ \text{real-valued}\ \mathbb{F}-\text{adapted c$\grave{a}$dl$\grave{%
a}$g process}:\mathbb{E}\Big[\sup\limits_{0\leq t\leq T}|\phi _{t}|^{2}\Big]%
<+\infty \bigg\}; \\
\mathcal{M}_{\mathbb{F}}^{2}(0,T;\mathbb{R}^{n}) &:&=\bigg\{(\phi
_{t})_{0\leq t\leq T}\ \mathbb{R}^{n}\text{-valued }\mathbb{F}-\text{adapted
process}:\mathbb{E}\Big[\int_{0}^{T}|\phi _{t}|^{2}dt\Big]<+\infty \bigg\}.
\end{eqnarray*}

\subsection{The classical FBSDEs}

We first recall some results on FBSDEs, for its proof the reader is referred
to Hu and Peng \cite{Hu Peng}. The FBSDEs they considered has the form
\begin{equation*}
\begin{array}{ll}
x_{t}=x_{0}+\int_{0}^{t}b(s,x_{s},y_{s},z_{s})ds+\int_{0}^{t}\sigma
(s,x_{s},y_{s},z_{s})dW_{s}, &  \\
&  \\
y_{t}=g(x_{T})+\int_{t}^{T}f(s,x_{s},y_{s},z_{s})ds-\int_{t}^{T}z_{s}dW_{s},%
\ \ \ t\in \lbrack 0,T].\label{Hu Peng FBSDE} &
\end{array}%
\end{equation*}%
Function $b,f:\Omega \times \lbrack 0,T]\times \mathbb{R}\times \mathbb{R}%
\times \mathbb{R}^{d}\rightarrow \mathbb{R}$, $\sigma :\Omega \times \lbrack
0,T]\times \mathbb{R}\times \mathbb{R}\times \mathbb{R}^{d}\rightarrow
\mathbb{R}^{d}$ with the property that $b(t,x,y,z)_{t\in \lbrack
0,T]},\sigma (t,x,y,z)_{t\in \lbrack 0,T]}$ and $f(t,x,y,z)_{t\in \lbrack
0,T]}$ are $\mathbb{F}$-progressively measurable for each $(x,y,z)\in
\mathbb{R}\times \mathbb{R}\times \mathbb{R}^{d}$.

Some notations and conditions are needed before giving the existence and
uniqueness of the solution of such FBSDEs. Let $<,>$ denote the usual inner
product in $\mathbb{R}^{n}$, and for $u=(x,y,z)\in \mathbb{R}\times \mathbb{R%
}\times \mathbb{R}^{d}$, we define
\begin{equation*}
F(t,u):=(-f(t,u),b(t,u),\sigma (t,u)).
\end{equation*}

\begin{enumerate}
\item[\textbf{(H1)}] (i) For each $u=(x,y,z)\in \mathbb{R}\times \mathbb{R}%
\times \mathbb{R}^{d}$, $F(\cdot ,u)\in \mathcal{M}^{2}(0,T;\mathbb{R}\times
\mathbb{R}\times \mathbb{R}^{d})$, and for each $x\in \mathbb{R},g(x)\in
L^{2}(\Omega ,\mathcal{F}_{T};\mathbb{R});$ there exists a constant $c_{1}>0$%
, such that
\begin{eqnarray*}
|F(t,u_{1})-F(t,u_{2})| &\leq &c_{1}|u_{1}-u_{2}|,\ P-a.s.,a.e.\ t\in
\mathbb{R}^{+}, \\
\ \ \ \forall u_{i} &\in &\mathbb{R}\times \mathbb{R}\times \mathbb{R}^{d}\
\ (i=1,2)\ ; \\
|g(x_{1})-g(x_{2})| &\leq &c_{1}|x_{1}-x_{2}|,\ P-a.s.,\ \ \ \forall
(x_{1},x_{2})\in \mathbb{R}\times \mathbb{R}.
\end{eqnarray*}%
(ii) There exists a constant $c_{2}>0$, such that
\begin{eqnarray*}
&<&F(t,u_{1})-F(t,u_{2}),u_{1}-u_{2}>\ \leq \ -c_{2}|u_{1}-u_{2}|^{2}, \\
&&\ \ \ \ \ \ \ \ \ \ \ \ P-a.s.,\ a.e.\ t\in \mathbb{R}^{+},\ \ \ \forall \
u_{i}\in \mathbb{R}\times \mathbb{R}\times \mathbb{R}^{d}\ (i=1,2)\ ; \\
&<&g(x_{1})-g(x_{2}),x_{1}-x_{2}>\ \geq \ c_{2}|x_{1}-x_{2}|^{2},\ P-a.s.,\
\ \ \forall \ (x_{1},x_{2})\in \mathbb{R}\times \mathbb{R}.
\end{eqnarray*}
\end{enumerate}

\begin{lemma}
Let assumptions (H1) hold, then there exists a unique adapted solution $%
(x,y,z)$ for the FBSDEs (1)
\end{lemma}

\subsection{Mean-field BSDEs and McKean-Vlasov SDEs}

This section is devoted to the recall of some basic results on a new type of
BSDEs, the so called mean-field BSDEs; the reader interested in more details
is referred to Buckdahn, Djehiche, Li, and Peng \cite{MFBSDE1} and Buckdahn
and Li et al. \cite{MFBSDE2}.

Let $(\bar{\Omega},\bar{\mathcal{F}},\bar{P})=(\Omega\times \Omega,\mathcal{F%
}\otimes \mathcal{F},P\otimes P)$ be the (non-completed) product of $(\Omega,%
\mathcal{F},P)$ with itself. We endow this product space with the filtration
$\bar{\mathbb{F}}=\{\bar{\mathcal{F}}_t=\mathcal{F}\otimes\mathcal{F}%
_t,0\leq t \leq T\}$. Any random variable $\xi \in L^0(\Omega,\mathcal{F},P)$
originally defined on $\Omega$ is extended canonically to $\bar{\Omega}%
:\xi^{\prime }(\omega^{\prime },\omega)=\xi(\omega^{\prime }),\
(\omega^{\prime },\omega)\in \bar{\Omega}=\Omega \times \Omega$. For any $%
\theta \in L^1(\bar{\Omega},\bar{\mathcal{F}},\bar{P})$ the variable $%
\theta(\cdot,\omega):\Omega \rightarrow \mathbb{R}$ belongs to $L^1(\Omega,%
\mathcal{F},P),P(d\omega)-a.s.$; we denote its expectation by
\begin{equation*}
\mathbb{E}^{\prime
}[\theta(\cdot,\omega)]=\int_{\Omega}\theta(\omega^{\prime
},\omega)P(d\omega^{\prime }).
\end{equation*}
Notice that $\mathbb{E}^{\prime }[\theta]=\mathbb{E}^{\prime 1}(\Omega,%
\mathcal{F},P)$, and
\begin{equation*}
\bar{E}[\theta]=\int_{\bar{\Omega}}\theta d\bar{P}=\int_{\Omega}\mathbb{E}%
^{\prime }[\theta(\cdot,\omega)]P(d\omega)=\mathbb{E}[\mathbb{E}^{\prime
}[\theta]].
\end{equation*}
The driver of mean-field BSDE is a function $f=f(\omega^{\prime },\omega,t,%
\tilde{y},\tilde{z},y,z):\bar{\Omega}\times[0,T]\times\mathbb{R}\times%
\mathbb{R}^d\times\mathbb{R}\times\mathbb{R}^d\rightarrow \mathbb{R}$ which
is $\bar{\mathbb{F}}$-progressively measurable for all $(\tilde{y},\tilde{z}%
,y,z)$, and satisfies the following assumptions.

\begin{enumerate}
\item[\textbf{(H2)}] (i) There exists a constant $C\geq 0$ such that, $\bar{P%
}$-a.s., for all $t\in \lbrack 0,T],y_{1},y_{2},\tilde{y}_{1},\tilde{y}%
_{2}\in \mathbb{R}$, $z_{1},z_{2},\tilde{z}_{1},\tilde{z}_{2}\in \mathbb{R}%
^{d},$ $|f(t,\tilde{y}_{1},\tilde{z}_{1},y_{1},z_{1})-f(t,\tilde{y}_{2},%
\tilde{z}_{2},y_{2},z_{2})|\leq C(|\tilde{y}_{1}-\tilde{y}_{2}|+|\tilde{z}%
_{1}-\tilde{z}_{2}|+|y_{1}-y_{2}|+|z_{1}-z_{2}|).$

(ii) $f(\cdot ,0,0,0,0)\in \mathcal{H}_{\bar{\mathbb{F}}}(0,T;\mathbb{R}).$
\end{enumerate}

The main result about mean-field BSDEs of Buckdahn and Li et al. \cite%
{MFBSDE2} is:

\begin{lemma}
Under the assumptions (H2), for any random variable $\xi \in L^{2}(\Omega ,%
\mathcal{F}_{T},P)$, the mean-field BSDEs
\begin{equation}
Y_{t}=\xi +\int_{t}^{T}\mathbb{E}^{\prime }[f(s,Y_{s}^{\prime
},Z_{s}^{\prime },Y_{s},Z_{s})]ds-\int_{t}^{T}Z_{s}dW_{s},\ \ \ \ 0\leq
t\leq T,  \label{MFBSDE}
\end{equation}%
has a unique adapted solution
\begin{equation*}
(Y_{t},Z_{t})\in \mathcal{S}_{\mathbb{F}}^{2}(0,T;\mathbb{R})\times \mathcal{%
M}_{\mathbb{F}}^{2}(0,T;\mathbb{R}^{d}).
\end{equation*}
\end{lemma}

\begin{remark}
The driving coefficient of (\ref{MFBSDE}) has to be interpreted as follows:
\begin{eqnarray*}
&&\mathbb{E}^{\prime }[f(s,Y_{s}^{\prime },Z_{s}^{\prime
},Y_{s},Z_{s})](\omega ) \\
&=&\mathbb{E}^{\prime }[f(s,Y_{s}(\omega ^{\prime }),Z_{s}(\omega ^{\prime
}),Y_{s}(\omega ),Z_{s}(\omega ))] \\
&=&\int_{\Omega }f(\omega ^{\prime },\omega ,s,Y_{s}(\omega ^{\prime
}),Z_{s}(\omega ^{\prime }),Y_{s}(\omega ),Z_{s}(\omega ))P(d\omega ^{\prime
}).
\end{eqnarray*}
\end{remark}

\noindent We shall also consider McKean-Vlasov SDEs (see, e.g., Buckdahn and
Li et al. \cite{MFBSDE2}). Let $b: \bar{\Omega} \times [0,T]\times \mathbb{R}
\times \mathbb{R}^d \rightarrow \mathbb{R}$ and $\sigma: \bar{\Omega} \times
[0,T]\times \mathbb{R} \times \mathbb{R}^d \rightarrow \mathbb{R}^{d}$ be
two measurable functions supposed to satisfy the following conditions:

\begin{enumerate}
\item[\textbf{(H3)}] (i) $b(\cdot ,\tilde{x},x)$ and $\sigma (\cdot ,\tilde{x%
},x)$ are $\bar{\mathbb{F}}$-progressively measurable continuous processes
for all $\tilde{x},x\in \mathbb{R}$, and there exists some constant $C>0$
such that
\begin{equation*}
|b(t,\tilde{x},x)|+|\sigma (t,\tilde{x},x)|\leq C(1+|\tilde{x}|+|x|),\ a.s.,
\end{equation*}%
for all $0\leq t\leq T$;

(ii) $b$ and $\sigma $ are Lipschitz in $\tilde{x},x$, i.e., there is some
constant $C>0$ such that
\begin{equation*}
|b(t,\tilde{x}_{1},x_{1})-b(t,\tilde{x}_{2},x_{2})|+|\sigma (t,\tilde{x}%
_{1},x_{1})-\sigma (t,\tilde{x}_{2},x_{2})|\leq C(|\tilde{x}_{1}-\tilde{x}%
_{2}|+|x_{1}-x_{2}|),a.s.
\end{equation*}%
for all $0\leq t\leq T$, $\tilde{x}_{1},\tilde{x}_{2},x_{1},x_{2}\in \mathbb{%
R}$.
\end{enumerate}

\noindent The McKean-Vlasov SDEs parameterized by the initial condition $%
(t,\zeta )\in \lbrack 0,T]\times L^{2}(\Omega ,\mathcal{F}_{t},P;\mathbb{R})$
is given as follows:
\begin{equation*}
\left\{
\begin{array}{ll}
dX_{s}^{t,\zeta }=\mathbb{E}^{\prime }\left[ b\left( s,\left( X_{s}^{t,\zeta
}\right) ^{\prime },X_{s}^{t,\zeta }\right) \right] ds+\mathbb{E}^{\prime }%
\left[ \sigma \left( s,\left( X_{s}^{t,\zeta }\right) ^{\prime
},X_{s}^{t,\zeta }\right) \right] dW_{s},\label{Mckean_Vlasov} &  \\
X_{t}^{t,\zeta }=\zeta ,\ \ \ s\in \lbrack t,T]. &
\end{array}%
\right.
\end{equation*}%
\noindent We recall that, due to our notational convention,
\begin{equation*}
\mathbb{E}^{\prime }\left[ b\left( s,\left( X_{s}^{t,\zeta }\right) ^{\prime
},X_{s}^{t,\zeta }\right) \right] (\omega )=\int_{\Omega }b\left( \omega
^{\prime },\omega ,s,X_{s}^{t,\zeta }\left( \omega ^{\prime }\right)
,X_{s}^{t,\zeta }\left( \omega \right) \right) P(d\omega ^{\prime }),\ \
\omega \in \Omega .
\end{equation*}

\begin{lemma}
Under Assumption (H3), SDEs () has a unique strong solution.
\end{lemma}

\begin{remark}
From standard arguments we also get that, for any $p\geq 2$, there exists $%
C_{p}\in \mathbb{R}$, which only depends on the Lipschitz and the growth
constants of $b$ and $\sigma $, such that for all $t\in \lbrack 0,T]$ and $%
\zeta ,\zeta ^{\prime p}(\Omega ,\mathcal{F}_{t},P;\mathbb{R})$,
\begin{eqnarray}
&&\mathbb{E}\Big[\sup\limits_{t\leq s\leq T}|X_{s}^{t,\zeta }-X_{s}^{t,\zeta
^{\prime }}|^{p}|\mathcal{F}_{t}\Big]\leq C_{p}|\zeta -\zeta ^{\prime p}|,\
\ \ \ a.s.,  \notag \\
&&\mathbb{E}\Big[\sup\limits_{t\leq s\leq T}|X_{s}^{t,\zeta }|^{p}|\mathcal{F%
}_{t}\Big]\leq C_{p}(1+|\zeta |^{p}),\ \ \ \ a.s.,
\label{Estimates for McKean-Vlasov} \\
&&\mathbb{E}\Big[\sup\limits_{t\leq s\leq t+\delta }|X_{s}^{t,\zeta }-\zeta
|^{p}|\mathcal{F}_{t}\Big]\leq C_{p}(1+|\zeta |^{p})\delta ^{\frac{p}{2}},
\notag
\end{eqnarray}%
P-a.s., for all $\delta >0$ with $t+\delta \leq T$ .

These, in the classical case, well-known standard estimates can be
consulted, for instance, in Ikeda and Watanabe \cite{Ikeda Watanabe}(pp.
166-168) and also in Karatzas and Shreve \cite{Karatzas}(pp. 289-290).
\end{remark}

\section{Fully coupled Mean-field FBSDEs}

In this section, we shall investigate a new type of FBSDEs called fully
coupled mean-field FBSDEs as follows:
\begin{equation*}
\begin{array}{ll}
X_{t}=X_{0}+\int_{0}^{t}\mathbb{E}^{\prime }[b(s,X_{s}^{\prime
},Y_{s}^{\prime },Z_{s}^{\prime },X_{s},Y_{s},Z_{s})]ds+\int_{0}^{t}\mathbb{E%
}^{\prime }[\sigma (s,X_{s}^{\prime },Y_{s}^{\prime },Z_{s}^{\prime
},X_{s},Y_{s},Z_{s})]dW_{s}, &  \\
&  \\
Y_{t}=\Phi (X_{T})+\int_{t}^{T}\mathbb{E}^{\prime }[f(s,X_{s}^{\prime
},Y_{s}^{\prime },Z_{s}^{\prime
},X_{s},Y_{s},Z_{s})]ds-\int_{t}^{T}Z_{s}dW_{s},\label{coupled FBSDE}\ \ \ \
t\in \lbrack 0,T]. &
\end{array}%
\end{equation*}%
Here the processes $X,Y,Z$ take values in $\mathbb{R},\mathbb{R},\mathbb{R}%
^{d}$ respectively; and $b,\sigma ,\Phi $ and $f$ take values in $\mathbb{R},%
\mathbb{R}^{d},\mathbb{R}$ and $\mathbb{R}$ respectively.

\begin{remark}
The driving coefficient here has to the same interpretation as Lemma 2:
\begin{eqnarray*}
\lefteqn{\mathbb{E}'[\psi(s, X'_s, Y'_s, Z'_s,X_s, Y_s,Z_s)](\omega)} \\
&=&\mathbb{E}^{\prime }[\psi(s,X_s(\omega^{\prime }), Y_s(\omega^{\prime }),
Z_s(\omega^{\prime }),X_s(\omega), Y_s(\omega),Z_s(\omega))] \\
&=&\int_{\Omega}\psi(\omega^{\prime },\omega, s, X_s(\omega^{\prime }),
Y_s(\omega^{\prime }), Z_s(\omega^{\prime }),X_s(\omega),
Y_s(\omega),Z_s(\omega))P(d\omega^{\prime }).
\end{eqnarray*}
for $\psi=b,\sigma,f.$
\end{remark}

For convenience, we will use the following notations in this section: Let $%
<,>$ denote the usual inner product in $\mathbb{R}^{n}$ and we use the usual
Euclidean norm in $\mathbb{R}^{n}$. For $\Theta =(\tilde{x},\tilde{y},\tilde{%
z},x,y,z)\in \mathbb{R}\times \mathbb{R}\times \mathbb{R}^{d}\times \mathbb{R%
}\times \mathbb{R}\times \mathbb{R}^{d},$
\begin{equation*}
F(t,\Theta )=(-f(t,\Theta ),b(t,\Theta ),\sigma (t,\Theta )).
\end{equation*}

Now we give the standard assumptions on the coefficients of mean-field FBSDE:

\begin{enumerate}
\item[\textbf{(H4)}] For each $\Theta \in \mathbb{R}\times \mathbb{R}\times
\mathbb{R}^{d}\times \mathbb{R}\times \mathbb{R}\times \mathbb{R}^{d},\
F(\cdot ,\Theta )\in \mathcal{M}^{2}(0,T;\mathbb{R}\times \mathbb{R}\times
\mathbb{R}^{d}\times \mathbb{R}\times \mathbb{R}\times \mathbb{R}^{d})$, and
for each $x\in \mathbb{R},\ g(x)\in L^{2}(\Omega ,\mathcal{F},\mathbb{R});$
there exists a constant $C>0$, such that:
\begin{eqnarray*}
&&|F(t,\Theta _{1})-F(t,\Theta _{2})|\leq C|\Theta _{1}-\Theta _{2}|,\ \
P-a.s.,\ a.e.\ t\in \mathbb{R}^{+}, \\
&&\ \ \ \ \ \Theta _{i}=(\tilde{x}_{i},\tilde{y}_{i},\tilde{z}%
_{i},x_{i},y_{i},z_{i})\in \mathbb{R}\times \mathbb{R}\times \mathbb{R}%
^{d}\times \mathbb{R}\times \mathbb{R}\times \mathbb{R}^{d},\ \ \ (i=1,2),
\end{eqnarray*}%
and
\begin{equation*}
|\Phi (x_{1})-\Phi (x_{2})|\leq C|x_{1}-x_{2}|,\ \ P-a.s.,\forall
(x_{1},x_{2})\in \mathbb{R}\times \mathbb{R}.
\end{equation*}
\end{enumerate}

The following monotone conditions are our main assumptions:

\begin{enumerate}
\item[\textbf{(H5)}] For $\Theta _{i}=(\tilde{x}_{i},\tilde{y}_{i},\tilde{z}%
_{i},x_{i},y_{i},z_{i})\in \mathbb{R}\times \mathbb{R}\times \mathbb{R}%
^{d}\times \mathbb{R}\times \mathbb{R}\times \mathbb{R}^{d}$, let $%
u_{i}=(x_{i},y_{i},z_{i})\in \mathbb{R}\times \mathbb{R}\times \mathbb{R}%
^{d} $, then $\Theta _{i}=(\tilde{u}_{i},u_{i})\ (i=1,2)$. We assume that
\begin{eqnarray*}
&&\mathbb{E}<F(t,\Theta _{1})-F(t,\Theta _{2}),u_{1}-u_{2}>\ \leq \ -C_{1}%
\mathbb{E}(|u_{1}-u_{2}|^{2})\ \ \ \ P-a.s.,a.e.\ t\in \mathbb{R}^{+}, \\
&<&\Phi (x_{1})-\Phi (x_{2}),x_{1}-x_{2}>\ \geq \ \mu
_{1}|x_{1}-x_{2}|^{2},\ \ P-a.s.,\forall (x_{1},x_{2})\in \mathbb{R}\times
\mathbb{R},
\end{eqnarray*}%
where $C_{1}$ and $\mu _{1}$ are given positive constants.
\end{enumerate}

For the mean-field FBSDE (\ref{coupled FBSDE}), we have the the following
main result of this section.

\begin{theorem}
Under the assumptions (H4) and (H5), there exists a unique adapted solution
(X,Y,Z) for mean-field FBSDEs (\ref{coupled FBSDE}).
\end{theorem}

The proof of this theorem is similar to that of Theorem 3.1 in \cite{Hu Peng}
except the mean-field term. However, to be self-contained, we intend to give
the proof. Before giving the proof of this theorem, we need the two
technical lemmas below whose proof will be given in the sequel.

\begin{lemma}
Suppose that $(\gamma(\cdot),\phi(\cdot),\varphi(\cdot))\in \mathcal{M}%
^2(0,T;\mathbb{R}\times \mathbb{R}^d\times \mathbb{R}),\ \xi \in L^2(\Omega,%
\mathcal{F}_T;\mathbb{R})$, then the following linear mean-field
forward-backward stochastic differential equations
\begin{eqnarray}
X_t&=&X_0+\int_0^t\big(-\mathbb{E}^{\prime }[Y^{\prime }_s]-Y_s+\gamma(s)%
\big)ds+\int_0^t\big(-\mathbb{E}^{\prime }[Z^{\prime }_s]-Z_s+\phi(s)\big)%
dW_s,  \label{lem 3.1 (1)} \\
Y_t&=&\xi+X_T+\int_t^T[\mathbb{E}^{\prime }[X^{\prime
}_s]+X_s-\varphi(s)]ds-\int_t^TZ_s dW_s,  \label{lem 3.1 (2)}
\end{eqnarray}
have a unique adapted solution: $(X,Y,Z)\in \mathcal{M}^2(0,T;\mathbb{R}%
\times \mathbb{R}\times \mathbb{R}^d).$
\end{lemma}

\noindent Now, we define, for any given $\alpha \in \mathbb{R}$,
\begin{eqnarray*}
&&b^{\alpha }(t,\tilde{x},\tilde{y},\tilde{z},x,y,z)=\alpha b(t,\tilde{x},%
\tilde{y},\tilde{z},x,y,z)+(1-\alpha )(-\tilde{y}-y), \\
&&\sigma ^{\alpha }(t,\tilde{x},\tilde{y},\tilde{z},x,y,z)=\alpha \sigma (t,%
\tilde{x},\tilde{y},\tilde{z},x,y,z)+(1-\alpha )(-\tilde{z}-z), \\
&&f^{\alpha }(t,\tilde{x},\tilde{y},\tilde{z},x,y,z)=\alpha f(t,\tilde{x},%
\tilde{y},\tilde{z},x,y,z)+(\alpha -1)(-\tilde{x}-x), \\
&&\Phi ^{\alpha }(x)=\alpha \Phi (x)+(1-\alpha )(x).
\end{eqnarray*}%
and consider the following equations:
\begin{eqnarray}
&&X_{t}=X_{0}+\int_{0}^{t}\big[\bar{b}^{\alpha }(s,\Lambda _{s})+\gamma (s)%
\big]ds+\int_{0}^{t}\big[\bar{\sigma}^{\alpha }(s,\Lambda _{s})+\phi (s)\big]%
dW_{s},  \label{Lem3.2(1)} \\
&&Y_{t}=\big(\Phi ^{\alpha }(X_{T})+\xi \big)+\int_{t}^{T}\big[\bar{f}%
^{\alpha }(s,\Lambda _{s})-\varphi (s)\big]ds-\int_{t}^{T}Z_{s}dW_{s},
\label{Lem3.2(2)}
\end{eqnarray}%
where we use the notation
\begin{equation*}
\Lambda _{s}=(X_{s}^{\prime },Y_{s}^{\prime },Z_{s}^{\prime
},X_{s},Y_{s},Z_{s}),
\end{equation*}%
and
\begin{equation*}
\bar{\psi}(s,\Lambda _{s})=\mathbb{E}^{\prime }[\psi (X_{s}^{\prime
},Y_{s}^{\prime },Z_{s}^{\prime },X_{s},Y_{s},Z_{s})],
\end{equation*}%
for $\psi =b,\sigma ,f.$ Then we can rewrite
\begin{eqnarray*}
\bar{b}^{\alpha }(s,\Lambda _{s}) &=&\alpha \bar{b}(s,\Lambda
_{s})+(1-\alpha )(-\mathbb{E}^{\prime }[Y_{s}^{\prime }]-Y_{s}), \\
\bar{\sigma}^{\alpha }(s,\Lambda _{s}) &=&\alpha \bar{\sigma}(s,\Lambda
_{s})+(1-\alpha )(-\mathbb{E}^{\prime }[Z_{s}^{\prime }]-Z_{s}), \\
\bar{f}^{\alpha }(s,\Lambda _{s}) &=&\alpha \bar{f}(s,\Lambda _{s})+(\alpha
-1)(-\mathbb{E}^{\prime }[X_{s}^{\prime }]-X_{s}).
\end{eqnarray*}

\begin{lemma}
For a given $\alpha_0 \in [0,1)$ and for any $(\gamma(\cdot),\phi(\cdot),%
\varphi(\cdot))\in \mathcal{M}^2(0,T; \mathbb{R} \times \mathbb{R}^d\times
\mathbb{R}),\ \xi\in L^2(\Omega,\mathcal{F}_T,P;\mathbb{R})$, assume that
Eqs (\ref{Lem3.2(1)}) and (\ref{Lem3.2(2)}) have an adapted solution. Then
there exists a $\delta_0 \in (0,1)$ which depends only on $c_1,c_2$ and $T$,
such that for all $\alpha \in [\alpha_0,\alpha_{0}+\delta_{0}]$, and for any
$(\gamma(\cdot),\phi(\cdot),\varphi(\cdot))\in \mathcal{M}^2(0,T; \mathbb{R}
\times \mathbb{R}^d\times \mathbb{R}),\ \xi\in L^2(\Omega,\mathcal{F}_T,P;%
\mathbb{R})$, Eqs (\ref{Lem3.2(1)}) and (\ref{Lem3.2(2)}) have an adapted
solution.
\end{lemma}

{\itshape Proof of Theorem 7.}

{\itshape Uniqueness.} \ If $U=(X,Y,Z)$ and $\bar{U}=(\bar{X},\bar{Y},\bar{Z}%
)$ are two adapted solutions of (\ref{coupled FBSDE}), we set
\begin{eqnarray*}
&&(\hat{X}^{\prime },\hat{Y}^{\prime },\hat{Z}^{\prime },\hat{X},\hat{Y},%
\hat{Z})=(X^{\prime }-\bar{X}^{\prime },Y^{\prime }-\bar{Y}^{\prime
},Z^{\prime }-\bar{Z}^{\prime },X-\bar{X},Y-\bar{Y},Z-\bar{Z}), \\
&&\hat{b}(t)=b(t,U^{\prime },U)-b(t,\bar{U}^{\prime },\bar{U}), \\
&&\hat{\sigma}(t)=\sigma (t,U^{\prime },U)-\sigma (t,\bar{U}^{\prime },\bar{U%
}), \\
&&\hat{f}(t)=f(t,U^{\prime },U)-f(t,\bar{U}^{\prime },\bar{U}).
\end{eqnarray*}%
From Assumption (A1), it follows that $\{\hat{X}_{t}\}$ and $\{\hat{Y}_{t}\}$
are continuous, and
\begin{equation*}
\mathbb{E}(\sup\limits_{t\in \lbrack 0,T]}|\hat{X}_{t}|^{2}+\sup\limits_{t%
\in \lbrack 0,T]}|\hat{Y}_{t}|^{2})<+\infty .
\end{equation*}%
Applying the Itô's formula to $\hat{X_{t}}\hat{Y_{t}}$ on $\left[ 0,T\right]
$, we have
\begin{eqnarray*}
\lefteqn{\mathbb{E}\lbrack \big(\Phi (X_{T})-\Phi (\bar{X}_{T})\big)\hat{X}%
_{T}]} \\
&=&\mathbb{E}\int_{0}^{T}\Big\{\mathbb{E}^{\prime }[\hat{b}(t)]\hat{Y}_{t}-%
\mathbb{E}^{\prime }[\hat{f}(t)]\hat{X}_{t}+\mathbb{E}^{\prime }[\hat{\sigma}%
(t)]\hat{Z}_{t}\Big\}dt. \\
&=&\mathbb{E}\int_{0}^{T}\Big\{<(-\mathbb{E}^{\prime }[\hat{f}(t)],\mathbb{E}%
^{\prime }[\hat{b}(t)],\mathbb{E}^{\prime }[\hat{\sigma}(t)]),(\hat{X}_{t},%
\hat{Y}_{t},\hat{Z}_{t})>\Big\}dt
\end{eqnarray*}%
By assumptions (A1) and (A2), we get then
\begin{equation*}
\mu _{2}|X_{T}-\bar{X}_{T}|^{2}\leq \mathbb{E}[(\Phi (X_{T})-\Phi (\bar{X}%
_{T}))\hat{X}_{T}]\leq -C_{1}\mathbb{E}\int_{0}^{T}|U-\bar{U}|^{2}dt.
\end{equation*}%
So, we get $U=\bar{U}$.

{\itshape Existence.} According to Lemma 8, we see immediately that, when $%
\alpha =0$, for any $(\gamma (\cdot ),\phi (\cdot ),\newline
\varphi (\cdot ))\in \mathcal{M}^{2}(0,T;\mathbb{R}\times \mathbb{R}%
^{d}\times \mathbb{R})$, $\xi \in L^{2}(\Omega ,\mathcal{F}_{T},P;\mathbb{R})
$, Eqs (\ref{Lem3.2(1)}) and (\ref{Lem3.2(2)}) have an adapted solution.
From Lemma 3.2, for any $(\gamma (\cdot ),\phi (\cdot ),\varphi (\cdot ))\in
\mathcal{M}^{2}(0,T;\mathbb{R}\times \mathbb{R}^{d}\times \mathbb{R})$, $\xi
\in L^{2}(\Omega ,\mathcal{F}_{T},P;\mathbb{R})$, we can solve Eqs (\ref%
{Lem3.2(1)}) and (\ref{Lem3.2(2)}) successively for the case $\alpha \in
\lbrack 0,\delta _{0}],[\delta _{0},2\delta _{0}],\cdots .$ When $\alpha =1$%
, for any $(\gamma (\cdot ),\phi (\cdot ),\varphi (\cdot ))\in \mathcal{M}%
^{2}(0,T;\mathbb{R}\times \mathbb{R}^{d}\times \mathbb{R})$, $\xi \in
L^{2}(\Omega ,\mathcal{F}_{T},P;\mathbb{R})$, the adapted solution of Eqs. (%
\ref{Lem3.2(1)}) and (\ref{Lem3.2(2)}) exists, then we deduce immediately
that the adapted solution of Eqs. (\ref{coupled FBSDE}) exists. \hfill $\Box
$

{\itshape Proof of Lemma 8}
\begin{proof}
We consider the following BSDEs:
\begin{eqnarray*}
\ddot{Y}_t=\xi+\int_t^T[-\mathbb{E}'[\ddot{Y}'_s]-\ddot{Y}_s-\varphi(s)+\gamma(s)]ds-\int_t^T(\mathbb{E}'[\ddot{Z}'_s]+2\ddot{Z}_s-\phi(s))dW_s.
\end{eqnarray*}
By Lemma 1, the above equation has a unique adapted solution
$(\ddot{Y},\ddot{Z})$.

Then we solve the following forward equation
\begin{eqnarray*}
X_t=x+\int_0^t\big(-\mathbb{E}'[X'_s]-X_s-\mathbb{E}'[\ddot{Y}'_s]-\ddot{Y}_s+\gamma(s)\big)ds+\int_0^t\big(-\mathbb{E}'[\ddot{Z}'_s]-\ddot{Z}_s+\phi(s)\big)dW_s,
\end{eqnarray*}
and set $ Y=\ddot{Y}+X,\ Z=\ddot{Z}$, we get
\begin{eqnarray*}
X_t&=&x+\int_0^t\big(-\mathbb{E}'[Y'_s]-Y_s+\gamma(s)\big)ds+\int_0^t\big(-\mathbb{E}'[Z'_s]-Z_s+\phi(s)\big)dW_s,\\
Y_t-X_t&=&\xi+\int_t^T[\mathbb{E}'[X'_s]-\mathbb{E}'[Y'_s]+X_s-Y_s-\varphi(s)+\gamma(s)]ds\\
&&-\int_t^T(\mathbb{E}'[Z'_s]+2Z_s-\phi(s))dW_s,\\
 X_T-X_t&=&\int_t^T\big(-\mathbb{E}'[Y'_s]-Y_s+\gamma(s)\big)ds+\int_t^T\big(-\mathbb{E}'[Z'_s]-Z_s+\phi(s)\big)dW_s.
\end{eqnarray*}
Then we have
\begin{equation*}
Y_t=\xi+X_T+\int_t^T[\mathbb{E}'[X'_s]+X_s-\varphi(s)]ds-\int_t^TZ_sdW_s.
\end{equation*}
So $(X, Y, Z)$ is a solution of Eqs.
 (\ref{lem 3.1 (1)}) and (\ref{lem 3.1 (2)}). Thus the existence is
 proved.

 As for uniqueness, it only has to use the method of the proof of
 uniqueness in Theorem 3.1 and we omit it.
\end{proof}{\itshape Proof of Lemma 9}
\begin{proof}
For simplicity, we set
\begin{eqnarray*}
U^{i}&=&(X^i,Y^i,Z^i),\\
\Lambda^0 &=& \big((X^0)',(Y^0)',(Z^0)',X^0,Y^0,Z^0\big)=0,\\
\Lambda^{i}&=&((X^i)',(Y^i)',(Z^i)',X^i,Y^i,Z^i),\\
\hat{\Lambda}^{i+1}&=&\Lambda^{i+1}-\Lambda^{i}=((\hat{X}^{i+1})',(\hat{Y}^{i+1})',(\hat{Z}^{i+1})',\hat{X}^{i+1},\hat{Y}^{i+1},\hat{Z}^{i+1}),\\
\hat{U}^{i+1}&=&U^{i+1}-U^{i}=(\hat{X}^{i+1},\hat{Y}^{i+1},\hat{Z}^{i+1}),
\end{eqnarray*}
 for all $i \in N^{+}$.

\noindent For any given $\alpha_0\in [0,1]$ and any $\delta>0$, we solve
iteratively the following equations:
\begin{eqnarray}
X^{i+1}_t&=&a+\int_0^t\Big(\bar{b}^{\alpha_0}(s,\Lambda^{i+1}_s)+\delta
\big[Y_s^i+\mathbb{E}'[(Y_s^{i})']+\bar{b}(s,\Lambda^{i}_s)\big]+\gamma(s)\Big)ds \nonumber\\
&&+\int_0^t\Big(\bar{\sigma}^{\alpha_0}(s,\Lambda^{i+1}_s)+\delta
\big[Z_s^i+\mathbb{E}'[(Z_s^{i})']+\bar{\sigma}(s,\Lambda^{i}_s)\big]+\phi(s)\Big)dW_s,
\label{Lem3.2(3)}
\end{eqnarray}
\begin{eqnarray}
Y^{i+1}_t&=&\big(\Phi^{\alpha_0}(X_T^{i+1})+\delta(\Phi(X_T^i)-X_T^i)+\xi\big)+\int_t^T\Big(\bar{f}^{\alpha_0}(s,\Lambda^{i+1}_s)\nonumber\\
&&+\delta
\big[\bar{f}(s,\Lambda^{i}_s)-\mathbb{E}'[(X_s^{i})']-X_s^i\big]-\varphi(s)\Big)ds
-\int_t^TZ^{i+1}_sdW_s. \label{Lem3.2(4)}
\end{eqnarray}
Applying the It$\hat{o}$ formula to
$\hat{X}_t^{i+1}\hat{Y}_t^{i+1}$, on and noticing that
$\mathbb{E}'[Y']=\mathbb{E}[Y]$, we have
\begin{eqnarray*}
\lefteqn{\mathbb{E}\big((\Phi^{\alpha_0}(X_T^{i+1})-\Phi^{\alpha_0}(X_T^{i}))\hat{X}_T^{i+1}\big)}\\
&=&-\delta\mathbb{E}\big[\big(\Phi(X_T^i)-\Phi(X_T^{i-1})-\hat{X}_T^i\big)\hat{X}_T^{i+1}\big]\\
&&+
\mathbb{E}\int^T_0\Bigg\{\hat{Y}_s^{i+1}[\bar{b}^{\alpha_0}(s,\Lambda^{i+1}_s)-\bar{b}^{\alpha_0}(s,\Lambda^{i}_s)]
-\hat{X}_s^{i+1}[\bar{f}^{\alpha_0}(s,\Lambda^{i+1}_s)-\bar{f}^{\alpha_0}(s,\Lambda^{i}_s)]\\
&&\ \
+\hat{Z}_s^{i+1}[\bar{\sigma}^{\alpha_0}(s,\Lambda^{i+1}_s)-\bar{\sigma}^{\alpha_0}(s,\Lambda^{i}_s)]\Bigg\}ds
+\delta\mathbb{E}\int_0^T\Bigg\{\hat{Y}_s^{i+1}\Big[\hat{Y}_s^i+\mathbb{E}[\hat{Y}_s^{i}]+\bar{b}(s,\Lambda^{i}_s)-\bar{b}(s,\Lambda^{i-1}_s)
\Big]\\
&&\ \
-\hat{X}_s^{i+1}\big[\bar{f}(s,\Lambda^{i}_s)-\bar{f}(s,\Lambda^{i-1}_s)-\hat{X}_s^i-\mathbb{E}[\hat{X}_s^{i}]
\big]+\hat{Z}_s^{i+1}\big[\hat{Z}_s^i+\mathbb{E}[\hat{Z}_s^{i}]+\bar{\sigma}(s,\Lambda^{i}_s)-\bar{\sigma}(s,\Lambda^{i-1}_s)
\big] \Bigg\}ds \\
&=&
-\delta\mathbb{E}[(\Phi(X_T^i)-\Phi(X_T^{i-1})-\hat{X}_T^i)\hat{X}_T^{i+1}]+
\mathbb{E}\int^T_0\left<\bar{F}^{\alpha_0}(s,\Lambda^{i+1}_s)-\bar{F}^{\alpha_0}(s,\Lambda^{i}_s),\hat{U}^{i+1}_s
\right>ds \\
&&+
\delta\mathbb{E}\int^T_0\left<\hat{U}^{i}_s+\mathbb{E}[\hat{U}^{i}_s]+\bar{F}(s,\Lambda^{i}_s)-\bar{F}(s,\Lambda^{i-1}_s),\hat{U}^{i+1}_s
\right>ds.
\end{eqnarray*}
Using the definition of $\bar{F}(t,\Lambda)$, we have the following
inequality:
\begin{eqnarray}
 \bar{F}(s,\Lambda^{i}_s)-\bar{F}(s,\Lambda^{i-1}_s)&=&\mathbb{E}'\big[F(s,\Lambda^{i}_s)-F(s,\Lambda^{i-1}_s)\big]
 \nonumber
 \\
 &\leq& C \mathbb{E}'|\Lambda^{i}_s-\Lambda^{i-1}_s| \nonumber\\
 &=&  C\big(\mathbb{E}[|\hat{U}^{i}|]+|\hat{U}^{i}|\big). \label{Lem3.2(5)}
\end{eqnarray}

\noindent From assumptions (A1),\ (A2), Eq (\ref{Lem3.2(5)}) and the notation
$$
\bar{F}^{\alpha_0}(s,\Lambda^{i+1}_s)=
\alpha_0\bar{F}(s,\Lambda^{i+1}_s)-(1-\alpha_0)\Big(\mathbb{E}[(\hat{U}_t^{i+1})']+\hat{U}_t^{i+1}\Big),$$
for any $\alpha_0\in [0,1]$, we deduce easily that
\begin{eqnarray*}
&&(\mu_1\alpha_0+1-\alpha_0)\mathbb{E}[\hat{X}_T^{i+1}]^2+(C_1\alpha_0+2-2\alpha_0)\mathbb{E}\int^T_0|\hat{U}_t^{i+1}|^2dt \\
& \leq & \delta\mathbb{E}[\hat{X}_T^i\hat{X}_T^{i+1}]
-\delta\mathbb{E}[(\Phi(X_T^i)-\Phi(X_T^{i-1}))\hat{X}_T^{i+1}]
+\delta\mathbb{E}\int^T_0<\hat{U}_t^{i},\hat{U}_t^{i+1}>dt\\
&&+\delta\mathbb{E}\int^T_0<\mathbb{E}[\hat{U}_t^{i}],\hat{U}_t^{i+1}>dt
+\delta\mathbb{E}\int^T_0<\bar{F}(t,\Lambda^{i}_t)-\bar{F}(t,\Lambda^{i-1}_t),\hat{U}_t^{i+1}>dt
\\
&\leq& 2\delta(1+C)\left(\mathbb{E}|\hat{X}_T^i||\hat{X}_T^{i+1}|+
\mathbb{E}\int^T_0|\hat{U}_t^{i}||\hat{U}_t^{i+1}|dt\right).
\end{eqnarray*}
In virtue of
$$\min\ (\mu_1\alpha_0+1-\alpha_0,C_1\alpha_0+2-2\alpha_0) \geq \mu_2=
\min\ (1,\mu_1,C_1)>0,$$ the above inequality yields
\begin{eqnarray*}
\mathbb{E}[\hat{X}_T^{i+1}]^2+\mathbb{E}\int^T_0|\hat{U}_t^{i+1}|^2dt
& \leq & \frac{2\delta(1+C)}{\mu_2}
\mathbb{E}\left(|\hat{X}_T^i||\hat{X}_T^{i+1}|+
\int^T_0|\hat{U}_t^{i}||\hat{U}_t^{i+1}|dt\right).
\end{eqnarray*}

\noindent For $\varepsilon=\frac{\mu_2}{4\delta(1+C)}>0$, with the help of $ab
\leq \frac{a^2}{4\varepsilon}+\varepsilon b^2,$ we can derive
\begin{equation}
\mathbb{E}[\hat{X}_T^{i+1}]^2+\mathbb{E}\int^T_0|\hat{U}_t^{i+1}|^2dt
 \leq  \left(\frac{2\delta(1+C)}{\mu_2}\right)^2
\mathbb{E}\left(|\hat{X}_T^{i}|^2+
\int^T_0|\hat{U}_t^{i}|^2dt\right). \label{lem3.2 (7)}
\end{equation}
Note that there exists a constant $C'>0$ which depends only on $C$
and $T$, such that
\begin{equation}
\mathbb{E}|\hat{X}_T^i|^2 \leq
C'\mathbb{E}\int_0^T\Big(|\hat{U}^{i}_s|^2+|\hat{U}^{i-1}_s|^2\Big)ds,\
\ \ \ \  \forall i \geq 1. \label{Lem3.2(8)}
\end{equation}
Indeed, for $i \geq 1,$
\begin{eqnarray*}
\mathbb{E}|\hat{X}_T^i|^2 &=&
2\mathbb{E}\int_0^T|\hat{X}_s^i|\Big\{\bar{b}^{\alpha_0}(s,\Lambda^{i}_s)-\bar{b}^{\alpha_0}(s,\Lambda^{i-1}_s)
+\delta\Big(\hat{Y}_s^{i-1}+\mathbb{E}'[(\hat{Y}_s^{i-1})']+\bar{b}(s,\Lambda^{i-1}_s)-\bar{b}(s,\Lambda^{i-2}_s)\Big)
\Big\}ds\\
&&+
\mathbb{E}\int_0^T\Big\{\bar{\sigma}^{\alpha_0}(s,\Lambda^{i}_s)-\bar{\sigma}^{\alpha_0}(s,\Lambda^{i-1}_s)
+\delta\Big(\hat{Z}_s^{i-1}+\mathbb{E}'[(\hat{Z}_s^{i-1})']+\bar{\sigma}(s,\Lambda^{i-1}_s)-\bar{\sigma}(s,\Lambda^{i-2}_s)\Big)
\Big\}^2ds \\
 &\leq &
\mathbb{E}\int_0^T|\hat{X}_s^i|^2dt+2\mathbb{E}\int_0^T\Big\{\big(\bar{b}^{\alpha_0}(s,\Lambda^{i}_s)-\bar{b}^{\alpha_0}(s,\Lambda^{i-1}_s)\big)^2+
\big(\bar{\sigma}^{\alpha_0}(s,\Lambda^{i}_s)-\bar{\sigma}^{\alpha_0}(s,\Lambda^{i-1}_s)\big)^2
\Big\}ds
\\
&&+2\delta^2\mathbb{E}\int_0^T\Big\{
\hat{Y}_s^{i-1}+\mathbb{E}'[(\hat{Y}_s^{i-1})']+\bar{b}(s,\Lambda^{i-1}_s)-\bar{b}(s,\Lambda^{i-2}_s)
\Big\}^2ds\\
&&+
2\delta^2\mathbb{E}\int_0^T\Big\{\hat{Z}_s^{i-1}+\mathbb{E}'[(\hat{Z}_s^{i-1})']+\bar{\sigma}(s,\Lambda^{i-1}_s)-\bar{\sigma}(s,\Lambda^{i-2}_s)
\Big\}^2ds
\\
 &\leq &
C\mathbb{E}\int_0^T\Big\{|\hat{U}^{i}_s|^2+|\hat{U}^{i-1}_s|^2
+\mathbb{E}\big[|\hat{Y}^{i-1}_s|^2+|\hat{Z}^{i-1}_s|^2\big] +
\mathbb{E}(|\hat{U}^{i}_s|^2)+
\mathbb{E}(|\hat{U}^{i-1}_s|^2)\Big\}ds.
\end{eqnarray*}
Similar to (\ref{Lem3.2(5)}), in the last inequality, we use the
fact that
\begin{eqnarray*}
|\bar{b}(s,\Lambda^{i-1}_s)-\bar{b}(s,\Lambda^{i-2}_s)|^2 & \leq &
\mathbb{E}'\Big(\big[b(s,\Lambda^{i-1}_s)-b(s,\Lambda^{i-2}_s)\big]^2\Big)\\
&\leq& c
\Big(\mathbb{E}(|\hat{U}^{i-1}_s|^2)+|\hat{U}^{i-1}_s|^2\Big),
\end{eqnarray*}
as well as
\begin{eqnarray*}
\bar{b}^{\alpha_0}(s,\Lambda^{i}_s)-\bar{b}^{\alpha_0}(s,\Lambda^{i-1}_s)
&\leq& \alpha_0
C|\Lambda^{i}_s-\Lambda^{i-1}_s|+(1-\alpha_0)(-\mathbb{E}[\hat{Y}^i]-\hat{Y}^i),
\end{eqnarray*}
where $c$ is a constant which depends on $C$. One can show that
$\bar{\sigma}(s,\Lambda^{i-1}_s)-\bar{\sigma}(s,\Lambda^{i-2}_s)$
and
$\bar{\sigma}^{\alpha_0}(s,\Lambda^{i}_s)-\bar{\sigma}^{\alpha_0}(s,\Lambda^{i-1}_s)$
have the similar results. By a standard method of estimation, the
desired result (\ref{Lem3.2(8)}) can be derived easily.

From (\ref{lem3.2 (7)})  and (\ref{Lem3.2(8)}) , we know that there
exists a constant $K >0 $ which depends only on $C,\ C_1,\ \mu_1$
and $T$, such that
\begin{equation*}
\mathbb{E}\int_0^T\Big|\hat{U}^{i+1}_s|^2ds \leq
K\delta^2\Bigg(\mathbb{E}\int_0^T\Big\{|\hat{U}^{i}_s|^2+|\hat{U}^{i-1}_s|^2\Big\}ds\Bigg).
\end{equation*}
Hence there exists a $\delta_0\in (0,1)$, which depends only on $C,\
C_1,\ \mu_1$ and $T$, such that when $0<\delta \leq \delta_0$,
\begin{equation*}
\mathbb{E}\int_0^T|\hat{U}^{i+1}_s|^2ds \leq
\frac{1}{4}\mathbb{E}\int_0^T|\hat{U}^{i}_s|^2ds+\frac{1}{8}\mathbb{E}\int_0^T|\hat{U}^{i-1}_s|^2ds.
\end{equation*}
That is
\begin{equation*}
\mathbb{E}\int_0^T\Big(|\hat{U}^{i+1}_s|^2+\frac{1}{4}|\hat{U}^{i}_s|^2\Big)ds
\leq
\frac{1}{2}\mathbb{E}\int_0^T\Big(|\hat{U}^{i}_s|^2+\frac{1}{4}|\hat{U}^{i-1}_s|^2\Big)ds.
\end{equation*}
Repeat the above inequality as many times as you desire, there holds
\begin{equation*}
\mathbb{E}\int_0^T\Big(|\hat{U}^{i+1}_s|^2+\frac{1}{4}|\hat{U}^{i}_s|^2\Big)ds
\leq
\left(\frac{1}{2}\right)^{i-1}\mathbb{E}\int_0^T\Big(|\hat{U}^{2}_s|^2+\frac{1}{4}|\hat{U}^{1}_s|^2\Big)ds,
\ \ i\geq 1.
\end{equation*}

\noindent It turns out that $U^i$ is a Cauchy sequence in $\mathcal
{M}^2(0,T;\mathbb{R}\times \mathbb{R}\times \mathbb{R}^{d})$ and its
limit is denoted by $U=(X,Y,Z)$. Passing to the limit in
Eqs.(\ref{Lem3.2(3)}) and (\ref{Lem3.2(4)}), we see that, when
$0<\delta \leq \delta_0, U=(X,Y,Z)$ solves Eqs.(\ref{Lem3.2(1)}) and
(\ref{Lem3.2(2)}) for $\alpha=\alpha_0+\delta$. The proof is
completed.
\end{proof}

The condition (A2) can be replaced by the following condition.

\begin{enumerate}
\item[\textbf{(H6)}] For $\Theta _{i}=(\tilde{x}_{i},\tilde{y}_{i},\tilde{z}%
_{i},x_{i},y_{i},z_{i})\in \mathbb{R}\times \mathbb{R}\times \mathbb{R}%
^{d}\times \mathbb{R}\times \mathbb{R}\times \mathbb{R}^{d}$, let $%
u_{i}=(x_{i},y_{i},z_{i})\in \mathbb{R}\times \mathbb{R}\times \mathbb{R}%
^{d} $, then $\Theta _{i}=(\tilde{u}_{i},u_{i})\ (i=1,2)$. We assume that
\begin{eqnarray*}
\mathbb{E}<F(t,\Theta _{1})-F(t,\Theta _{2}),u_{1}-u_{2}> &\geq &C_{1}%
\mathbb{E}(|u_{1}-u_{2}|^{2}),\ \ P-a.s.,a.e.\ t\in \mathbb{R}^{+}, \\
<\Phi (x_{1})-\Phi (x_{2}),x_{1}-x_{2}> &\leq &-\mu _{1}|x_{1}-x_{2}|^{2},\
\ \ \ P-a.s.,\forall \ (x_{1},x_{2})\in \mathbb{R}\times \mathbb{R},
\end{eqnarray*}%
where $C_{1}$ and $\mu _{1}$ are given positive constants.
\end{enumerate}

We have another parallel existence and uniqueness theorem for mean-field
FBSDEs.

\begin{theorem}
Let (H4) and (H6) hold. Then there exists a unique adapted solution (X,Y,Z)
of mean-field FBSDEs (\ref{coupled FBSDE}).
\end{theorem}

The method to prove the existence is similar to Theorem 7. We now consider
the following (\ref{FBSDE (A3)}) for each $\alpha \in \lbrack 0,1]:$
\begin{eqnarray}
&&\ \ dX_{s}^{\alpha }=\big[\alpha \bar{b}(s,\Lambda _{s})+\gamma (s)\big]ds+%
\big[\alpha \bar{\sigma}(s,\Lambda _{s})+\phi (s)\big]dW_{s},  \notag \\
&&-dY_{s}=\big[-(1-\alpha )c_{2}X_{s}+\alpha \bar{f}(s,\Lambda _{s})+\varphi
(s)\big]ds-Z_{s}dW_{s},  \label{FBSDE (A3)} \\
&&\ \ X_{0}^{\alpha }=a,\ Y_{T}^{\alpha }=\alpha \Phi (X_{T})-(1-\alpha
)X_{T}+\xi ,  \notag
\end{eqnarray}%
where $\gamma ,\ \phi $ and $\varphi $ are given processes in $\mathcal{M}%
^{2}(0,T)$ with values in $\mathbb{R},\ \mathbb{R}^{d}$, and $\mathbb{R}$,
resp., $\xi \in L^{2}(\Omega ,\mathcal{F}_{T},P)$. Clearly the existence of (%
\ref{FBSDE (A3)}) for $\alpha =1$ implies the existence of FBSDEs (\ref%
{coupled FBSDE}). From the existence and uniqueness of SDEs and BSDEs, when $%
\alpha =0$, the equation (\ref{FBSDE (A3)}) has a unique solution.

In order to obtain this conclusion, we also need the following lemma. This
lemma gives a priori estimate for the existence interval of (\ref{FBSDE (A3)}%
) with respect to $\alpha \in [0,1].$

\begin{lemma}
We assume (H4) and (H6). Then there exists a positive constant $\delta _{0}$
such that if, a prior, for a $\alpha _{0}\in \lbrack 0,1)$ there exists a
triple of solution $(X^{\alpha _{0}},Y^{\alpha _{0}},Z^{\alpha _{0}})$ of (%
\ref{FBSDE (A3)}), then for each $\delta \in \lbrack 0,\delta _{0}]$ there
exists a solution $(X^{\alpha _{0}+\delta },Y^{\alpha _{0}+\delta
},Z^{\alpha _{0}+\delta })$ of FBSDEs (\ref{FBSDE (A3)}) for $\alpha =\alpha
_{0}+\delta $.
\end{lemma}

\begin{proof}
We use the notations
\begin{eqnarray*}
&&\ u=(x,y,z),  \ \ \ \ \ \ \ \ \ \ \ \ \ \ \ \ \ \ \ \ \bar{u}=(\bar{x},\bar{y},\bar{z}),\\
&&U=(X,Y,Z),  \ \ \ \ \ \ \ \ \  \ \ \ \ \ \ \ \ \ \bar{U}=(\bar{X},\bar{Y},\bar{Z}),\\
&& \Theta=(x',y',z',x,y,z),\ \ \ \ \ \ \ \ \
\bar{\Theta}=(\bar{x}',\bar{y}',\bar{z}',\bar{x},\bar{y},\bar{z}),\\
&&\Lambda=(X',Y',Z',X,Y,Z),
\ \ \ \ \ \bar{\Lambda}=(\bar{X}',\bar{Y}',\bar{Z}',\bar{X},\bar{Y},\bar{Z}),\\
&&\hat{\Theta}=\Theta-\bar{\Theta},\ \ \ \ \ \ \ \ \ \ \ \ \ \ \ \ \
\ \ \ \ \ \hat{\Lambda}=\Lambda-\bar{\Lambda},\\ &&
\hat{u}=(\hat{x},\hat{y},\hat{z})=(x-\bar{x},y-\bar{y},z-\bar{z}),\\
&&
\hat{U}=(\hat{X},\hat{Y},\hat{Z})=(X-\bar{X},Y-\bar{Y},Z-\bar{Z}).
\end{eqnarray*}

\noindent Since for $(\gamma,\phi,\varphi) \in \mathcal {M}^2(0,T;
\mathbb{R}\times \mathbb{R}^d \times \mathbb{R})$, $\xi \in
L^2(\Omega, \mathcal{F}_T,P)$, $\alpha_0 \in [0,1)$ there exists a
unique solution of (\ref{FBSDE (A3)}), thus, for each $x_T \in
L^2(\Omega, \mathcal{F}_T,P)$ and a triple $u_s=(x_s,y_s,z_s) \in
\mathcal {M}^2(0,T; \mathbb{R}\times \mathbb{R}^d \times
\mathbb{R})$ there exists a unique triple $U_s=(X_s,Y_s,Z_s) \in
\mathcal {M}^2(0,T; \mathbb{R}\times \mathbb{R}^d \times
\mathbb{R})$ satisfying the following FBSDE:
\begin{eqnarray}
&&\ \ dX_s=\big[\alpha_0\bar{b}(s,\Lambda_s)+\delta
\bar{b}(s,\Theta_s)+\gamma(s)\big]ds+\big[\alpha_0\bar{\sigma}(s,\Lambda_s)+\delta
\bar{\sigma}(s,\Theta_s)+\phi(s)\big]dW_s,\nonumber
\\&&-dY_s=\big[-(1-\alpha_0)C_1X_s+\alpha_0\bar{f}(s,\Lambda_s)+\delta\big(C_1x_s+\bar{f}(s,\Theta_s)\big)+\varphi(s)\big]ds-Z_sdW_s, \label{FBSDE (A3) mapping}\\
&&\ \ X_0=a,\
Y_T=\alpha_0\Phi(X_T)+(\alpha_0-1)X_T+\delta(\Phi(x_T)+x_T)+\xi.
\nonumber
\end{eqnarray}
We now proceed to prove that, if $\delta$ is sufficiently small, the
mapping defined by
\begin{eqnarray*}
&&\ \ \ \ \ \ \ \ \ \ \ I_{\alpha_0+\delta}(u\times x_T)=U\times X_T:\\
&&\mathcal{M}^2(0,T; \mathbb{R}\times \mathbb{R}^d \times
\mathbb{R}) \times L^2(\Omega, \mathcal{F}_T,P)\rightarrow \mathcal
{M}^2(0,T; \mathbb{R}\times \mathbb{R}^d \times \mathbb{R}) \times
L^2(\Omega, \mathcal{F}_T,P)
\end{eqnarray*}
is a contraction.

Let $\bar{u}=(\bar{x},\bar{y},\bar{z}) \in \mathcal {M}^2(0,T;
\mathbb{R}\times \mathbb{R}^d \times \mathbb{R})$, and let $\bar{U}
\times \bar{X}_T = I_{\alpha_0+\delta}(\bar{u}\times \bar{x}_T).$

Using It\textrm{$\hat{o}$}'s formula to $\hat{X}_s\hat{Y}_s$ yields
\begin{eqnarray*}
&&\alpha_0\mathbb{E}<\Phi(X_T)-\Phi(\bar{X}_T),\hat{X}_T>+(\alpha_0-1)\mathbb{E}|\hat{X}_T|^2+\delta\mathbb{E}<\Phi(x_T)-\Phi(\bar{x}_T)+\hat{x}_T,\hat{X}_T>\\
&=&-\mathbb{E}\int_0^T\hat{X}_s\Big[\alpha_0[\bar{f}(s,\Lambda_s)-\bar{f}(s,\bar{\Lambda}_s)]+\delta[\bar{f}(s,\Theta_s)-\bar{f}(s,\bar{\Theta}_s)]
-(1-\alpha_0)C_1\hat{X}_s+\delta
C_1\hat{x}_s\Big]ds\\
&&+\mathbb{E}\int_0^T\hat{Y}_s\Big[\alpha_0[\bar{b}(s,\Lambda_s)-\bar{b}(s,\bar{\Lambda}_s)]+\delta[\bar{b}(s,\Theta_s)-\bar{b}(s,\bar{\Theta}_s)]\Big]ds\\
&&+\mathbb{E}\int_0^T\hat{Z}_s\Big[\alpha_0[\bar{\sigma}(s,\Lambda_s)-\bar{\sigma}(s,\bar{\Lambda}_s)]+\delta[\bar{\sigma}(s,\Theta_s)-\bar{\sigma}(s,\bar{\Theta}_s)]\Big]ds\\
&=&\alpha_0\mathbb{E}\int_0^T<\bar{F}(s,\Lambda_s)-\bar{F}(s,\bar{\Lambda}_s),\hat{U}_s>ds+(1-\alpha_0)C_1\mathbb{E}\int_0^T|\hat{X}_s|^2ds-\delta
C_1\mathbb{E}\int_0^T<\hat{X}_s,
\hat{x}_s>ds \\
&&+\delta
\mathbb{E}\int_0^T<\bar{F}(s,\Theta_s)-\bar{F}(s,\bar{\Theta}_s),\hat{U}_s>
  ds.
\end{eqnarray*}

\noindent From (H4) and (H6), we can get
\begin{eqnarray*}
&& (\mu_1\alpha_0+(1-\alpha_0))\mathbb{E}[|\hat{X}_T|^2]+C_1\mathbb{E}\int_0^T|\hat{X}_s|^2ds+C_1\alpha_0\mathbb{E}\int_0^T(|\hat{Y}_s|^2+|\hat{Z}_s|^2)ds\\
& \leq &
\delta\mathbb{E}<\Phi(x_T)-\Phi(\bar{x}_T)+\hat{x}_T,\hat{X}_T>
+\delta C_1\mathbb{E}\int^T_0<\hat{X}_s,\hat{x}_s>ds\\&&-\delta
\mathbb{E}\int_0^T<\bar{F}(s,\Theta_s)-\bar{F}(s,\bar{\Theta}_s),\hat{U}_s>
  ds \\
& \leq & \delta K_1 \mathbb{E}(|\hat{x}_T|^2+|\hat{X}_T|^2)+\delta
K_1\mathbb{E}\int^T_0\Big(|\hat{u}_s|^2+|\hat{U}_s|^2\Big)ds.
\end{eqnarray*}
This means
\begin{eqnarray*}
&&\mu\mathbb{E}[|\hat{X}_T|^2]+C_1\mathbb{E}\int_0^T|\hat{X}_s|^2ds
\leq  \delta K_1 \mathbb{E}(|\hat{x}_T|^2+|\hat{X}_T|^2)+\delta
K_1\mathbb{E}\int^T_0\Big(|\hat{u}_s|^2+|\hat{U}_s|^2\Big)ds,
\end{eqnarray*}
where $\mu_1\alpha_0+(1-\alpha_0)\geq \mu=\min(1,\mu_1)>0$.

 On the other hand, for the difference
of the solutions $(\hat{Y},\hat{Z}) = (Y-\bar{Y}, Z-\hat{Z})$, we
apply the usual technique to the BSDE part:
\begin{eqnarray*}
\mathbb{E}\int_0^T\Big(|\hat{Y}_s|^2+|\hat{Z}_s|^2\Big)ds \leq K_2
\delta \mathbb{E}\int_0^T|\hat{u}_s|^2ds +K_2 \delta
\mathbb{E}|\hat{x}_T|^2 +K_2 \mathbb{E}\int_0^T|\hat{X}_s|^2ds +C_1
\mathbb{E}|\hat{X}_T|^2.
\end{eqnarray*}
Here the constant $K_2$ depends on the Lipschitz constants $C, C_1,
$ and $T$.

Combining the above two estimates, it is clear that, we have
\begin{equation*}
\mathbb{E}[|\hat{X}_T|^2]+\mathbb{E}\int_0^T|\hat{U}_s|^2ds\\
 \leq  \delta
K\mathbb{E}\Big(\int^T_0|\hat{u}_s|^2ds+|\hat{x}_T|^2\Big).
\end{equation*}

\noindent Here the constant $K$ depends only on $C,\ C_1,\ \mu_1,\ K_1$, and
$T$. We now choose $\delta_0 =\frac{1}{2K}$. It is clear that, for
each fixed $\delta \in [0,\delta_0]$, the mapping
$I_{\alpha_0+\delta}$ is a contraction in the sense that
\begin{equation*}
\mathbb{E}[|\hat{X}_T|^2]+\mathbb{E}\int_0^T|\hat{U}_s|^2ds\\
 \leq
\frac{1}{2}\Big(\mathbb{E}\int^T_0|\hat{u}_s|^2ds+\mathbb{E}|\hat{x}_T|^2\Big).
\end{equation*}

\noindent It follows that this mapping has a unique fixed point
$U^{\alpha_0+\delta}=(X^{\alpha_0+\delta},Y^{\alpha_0+\delta},Z^{\alpha_0+\delta})
$ which is the solution of (\ref{FBSDE (A3)}) for $\alpha =\alpha_0
+ \delta$. The proof is complete.
\end{proof}

We now give the proof of Theorem 10.

{\itshape Proof of Theorem 10.} The uniqueness is obvious from Theorem 7.
When $\alpha =0,$ the equation (\ref{FBSDE (A3)}) has a unique solution. It
then follows from Lemma 3.3 that there exists a positive constant $\delta
_{0}$ depending on Lipschitz constants $C,\ C_{1},\ \mu _{1},\ K_{1}$ and $T$
such that, for each $\delta \in \lbrack 0,\delta _{0}]$, equation (\ref%
{FBSDE (A3)}) for $\alpha =\alpha _{0}+\delta $ has a unique solution. We
can repeat this process for $N$-times with $1\leq N\delta _{0}<1+\delta _{0}$%
. It then follows that, in particular, FBSDEs (\ref{FBSDE (A3)}) for $\alpha
=1$ with $\xi =0$ has a unique solution. The proof is complete.\ \ \ \ \ \
\hfill $\Box $

Theorem 10 can ensure the existence and uniqueness of solution to the
adjoint forward-backward systems in Section 5 and Section 7.

\section{Stochastic maximum principle in mean-field controls of FBSDEs}

In this section, we study the stochastic maximum principle for mean-field
control problem of FBSDEs. The action space $U$ is a non-empty, closed and
convex subset of $\mathbb{R}^k \ (k \in \mathbb{N}^{+})$, and we define the
admissible control set as
\begin{equation*}
\mathcal{U}=\{v_t \in L^2_{\bar{\mathbb{F}}}(0,T;U)|v_t(\omega^{\prime
},\omega):[0,T] \times \Omega \times \Omega \rightarrow U , t\in [0,T]\}.
\end{equation*}

\noindent For any $v(\cdot )\in \mathcal{U}$, we consider the following
forward-backward stochastic control system of Mean-field type:
\begin{equation*}
\left\{
\begin{array}{ll}
dX_{t}=\mathbb{E}^{\prime }[b(t,X_{t}^{\prime },X_{t},v_{t})]dt+\mathbb{E}%
^{\prime }[\sigma (t,X_{t}^{\prime },X_{t},v_{t})]dW_{t}, &  \\
X(0)=x_{0}, &  \\
-dY_{t}=\mathbb{E}^{\prime }[f(t,X_{t}^{\prime },Y_{t}^{\prime
},Z_{t}^{\prime },X_{t},Y_{t},Z_{t},v_{t})]dt-Z_{t}dW_{t},\label{state
equation for control} &  \\
Y_{T}=\Phi (X_{T}), &
\end{array}%
\right.
\end{equation*}%
where
\begin{equation*}
\begin{array}{ll}
b:[0,T]\times \mathbb{R}\times \mathbb{R}\times U\rightarrow \mathbb{R}, &
\\
\sigma :[0,T]\times \mathbb{R}\times \mathbb{R}\times U\rightarrow \mathbb{R}%
^{d}, &  \\
f:[0,T]\times \mathbb{R}\times \mathbb{R}\times \mathbb{R}^{d}\times \mathbb{%
R}\times \mathbb{R}\times \mathbb{R}^{d}\times U\rightarrow \mathbb{R}, &
\\
\Phi :\mathbb{R}\rightarrow \mathbb{R}. &
\end{array}%
\end{equation*}

\noindent The optimal control problem is to minimize the following expected
cost functional over $\mathcal{U}$:
\begin{eqnarray}
J(v(\cdot))&=&\mathbb{E}\left(\int_0^T\mathbb{E}^{\prime }[h(t,X^{\prime
}_t,Y^{\prime }_t,Z^{\prime }_t,X_t,Y_t,Z_t,v(t))]dt\right)  \notag \\
&&+\mathbb{E}\Big(g(X_T)+\gamma(Y(0))\Big),
\label{cost functional for control}
\end{eqnarray}
where
\begin{eqnarray*}
\begin{array}{ll}
g:\mathbb{R}\rightarrow \mathbb{R}, &  \\
\gamma:\mathbb{R}\rightarrow \mathbb{R}, &  \\
h: [0,T]\times \mathbb{R} \times \mathbb{R}\times \mathbb{R}^k\times \mathbb{%
R}\times \mathbb{R}\times \mathbb{R}^k \times U \rightarrow \mathbb{R}. &
\end{array}%
\end{eqnarray*}

\noindent An admissible control $u \in \mathcal{U}$ is said to be optimal if
\begin{equation}
J(u)=\min\limits_{v \in \mathcal{U}}J(v).  \label{optimal control}
\end{equation}

\noindent Now we give the following conditions in this section.

\begin{enumerate}
\item[\textbf{(A1)}] The given functions $b(t,\tilde{x},x,v),\sigma (t,%
\tilde{x},x,v),f(t,\tilde{x},\tilde{y},\tilde{z},x,y,z,v),h(t,\tilde{x},%
\tilde{y},\tilde{z},x,y,z,v),g(x)$ and $\gamma (y)$ are continuously
differentiable with respect to all of their components respectively.

\item[\textbf{(A2)}] All the derivatives in (A1) are Lipschitz continuous
and bounded.
\end{enumerate}

For any admissible controls $v(\cdot )\in \mathcal{U}$, due to Lemma 2, the
mean-field FBSDEs (\ref{state equation for control}) admits a unique
solution under assumptions (A1) and (A2), which is denoted by $%
(X_{t},Y_{t},Z_{t})$.

\subsection{Variational equations and variational inequality}

Let $u(\cdot)$ be an optimal control and $(X^u(\cdot),Y^u(\cdot),Z^u(\cdot))$
be the corresponding state trajectory of stochastic control system. For any $%
0 \leq \theta \leq 1$, we denote by $(X^{\theta}_t,Y^{\theta}_t,Z^{%
\theta}_t) $ the state trajectory corresponding the following perturbation $%
u^{\theta}_t $ of $u_t$.
\begin{equation*}
u^{\theta}_t=u_t+\theta(v_t-u_t),\ \ \ \ v_t \in \mathcal{U}.
\end{equation*}
Since $\mathcal{U}$ is convex, then $u^{\theta}(\cdot)$ is also in $\mathcal{%
U}$. Let $(k(\cdot),m(\cdot),n(\cdot))$ be a solution of the variational
equation
\begin{eqnarray}
&&\left\{
\begin{array}{ll}
dk_t=\mathbb{E}^{\prime }\Big[b_{\tilde{x}}(t,(X^u_t)^{\prime
u}_t,u_t)(k_t)^{\prime }+b_x(t,(X^u_t)^{\prime u}_t,u_t)k_t &  \\
\ \ \ \ \ \ \ \ \ \ \ \ +b_{v}(t,(X^u_t)^{\prime u}_t,u_t)(v_t-u_t)\Big]dt +%
\mathbb{E}^{\prime }\Big[\sigma_{\tilde{x}}(t,(X^u_t)^{\prime
u}_t,u_t)(k_t)^{\prime } &  \\
\ \ \ \ \ \ \ \ \ \ \ \ +\sigma_x(t,(X^u_t)^{\prime
u}_t,u_t)k_t+\sigma_{v}(t,(X^u_t)^{\prime },X_t,u_t)(v_t-u_t)\Big]dW_t, %
\label{variation equation 1 for control} &  \\
k_{0}=0, &
\end{array}
\right. \\
&&\left\{
\begin{array}{ll}
dm_t=-\mathbb{E}^{\prime }[\bar{f}_{\tilde{x}}(t,u)(k_t)^{\prime }+\bar{f}%
_{x}(t,u)k_t+\bar{f}_{\tilde{y}}(t,u)(m_t)^{\prime }+\bar{f}_{y}(t,u)m_t +%
\bar{f}_{\tilde{z}}(t,u)(n_t)^{\prime } &  \\
\label{variation equation 2 for control}\ \ \ \ \ \ \ \ \ \ \ \ \ \ +\bar{f}%
_{z}(t,u)n_t+\bar{f}_{v}(t,u)(v_t-u_t)]dt +n_tdW_t, &  \\
m_T=k_T\Phi_x(X_T). &
\end{array}
\right.
\end{eqnarray}
where we use the notation $\bar{f}(t,u)=f(t,(X^u_t)^{\prime u}_t)^{\prime
u}_t)^{\prime u}_t,Y^u_t,Z^u_t,u_t).$

Set
\begin{eqnarray}
\tilde{X}^\theta_t&=&\theta^{-1}(X^\theta_t-X^u_t)-k_t,  \notag \\
\tilde{Y}^\theta_t&=&\theta^{-1}(Y^\theta_t-Y^u_t)-m_t, \\
\tilde{Z}^\theta_t&=&\theta^{-1}(Z^\theta_t-Z^u_t)-n_t.  \notag
\end{eqnarray}

\noindent Then, we have the following convergence result:

\begin{lemma}
We suppose (A3) and (A4) hold. Then
\begin{eqnarray}
\lim \limits_{\theta \rightarrow 0}\sup\limits_{0\leq t\leq T}\mathbb{E}|%
\tilde{X}^\theta_t|^2=0,  \notag \\
\lim \limits_{\theta \rightarrow 0}\sup\limits_{0\leq t\leq T}\mathbb{E}|%
\tilde{Y}^\theta_t|^2=0,  \label{convergence result for
control} \\
\lim \limits_{\theta \rightarrow 0}\sup\limits_{0\leq t\leq T}\mathbb{E}|%
\tilde{Z}^\theta_t|^2=0.  \notag
\end{eqnarray}
\end{lemma}

\begin{proof}
Since the coefficients in linear mean-field FBSDE (\ref{variation
equation 1 for control}) and (\ref{variation equation 2 for
control}) are bounded, it follows from Proposition 1.2 in
\cite{MFBSDE2} that there exists a unique solution
$(k(t),m(t),n(t))$ for equations (\ref{variation equation 1 for
control}) and (\ref{variation equation 2 for control}).

The proof for the convergence of $\tilde{X}^\theta_t$ can be found
in Lemma 3.2 of \cite{Juan Li Automatica}. We need only to deal with
$\tilde{Y}^\theta_t$ and $\tilde{Z}^\theta_t$. From the definition
of $\tilde{Y}^\theta_t$, it fulfills the following BSDE,
\begin{eqnarray*}
-d\tilde{Y}^\theta_t &=&-\frac{1}{\theta}(dY^\theta_t-dY^u_t)+dm_t \\
                     &=&\frac{1}{\theta}\mathbb{E}'[f(t,(X^{\theta}_t)',(Y^{\theta}_t)',(Z^{\theta}_t)',X^{\theta}_t,Y^{\theta}_t,Z^{\theta}_t,u^{\theta}_t)-\bar{f}(t,u)]dt\\
                     &&-\mathbb{E}'[\bar{f}_{\tilde{x}}(t,u)(k_t)'+\bar{f}_{x}(t,u)k_t+\bar{f}_{\tilde{y}}(t,u)(m_t)'+\bar{f}_{y}(t,u)m_t+\bar{f}_{\tilde{z}}(t,u)(n_t)'+\bar{f}_{z}(t,u)n_t\\
                     &&\ \ \ \ \ \ \ +\bar{f}_{v}(t,u)(v_t-u_t)]dt-\tilde{Z}^\theta_tdW_t.
\end{eqnarray*}

\noindent Denote by
$X^{\lambda,\theta}_t=X^u_t+\lambda\theta(\tilde{X}^\theta_t+k_t),Y^{\lambda,\theta}_t=Y^u_t+\lambda\theta(\tilde{Y}^\theta_t+m_t),Z^{\lambda,\theta}_t=Z^u_t+\lambda\theta(\tilde{Z}^\theta_t+n_t)$
and $u^{\lambda,\theta}(t)=u_t+\lambda\theta (v_t-u_t)$. For
convenience, we introduce the notation
\begin{eqnarray*}
f(\lambda)=f(t,(X^{\lambda,\theta}_t)',(Y^{\lambda,\theta}_t)',(Z^{\lambda,\theta}_t)',X^{\lambda,\theta}_t,Y^{\lambda,\theta}_t,Z^{\lambda,\theta}_t,u^{\lambda,\theta}(t)).
\end{eqnarray*}
Then, we have
\begin{eqnarray*}
\lefteqn{\frac{1}{\theta}\left(f(t,(X^{\theta}_t)',(Y^{\theta}_t)',(Z^{\theta}_t)',X^{\theta}_t,Y^{\theta}_t,Z^{\theta}_t,u^{\theta}_t)-\bar{f}(t,u)\right)dt}
\\
&=&\int_0^1f_{\tilde{x}}(\lambda)((\tilde{X}^\theta_t)'+(k_t)')d\lambda+\int_0^1f_x(\lambda)(\tilde{X}^\theta_t+k_t)d\lambda
+\int_0^1f_{\tilde{y}}(\lambda)((\tilde{Y}^\theta_t)'+(m_t)')d\lambda\\
&&+\int_0^1f_y(\lambda)(\tilde{Y}^\theta_t+m_t)d\lambda+\int_0^1f_{\tilde{z}}(\lambda)((\tilde{Z}^\theta_t)'+(n_t)')d\lambda+\int_0^1f_z(\lambda)(\tilde{Z}^\theta_t+n_t)d\lambda
\\
&&+\int_0^1f_{v}(\lambda)(v_t-u_t)d\lambda,
\end{eqnarray*}
and $\tilde{Y}^\theta_t$ satisfies
\begin{eqnarray}
-d\tilde{Y}^\theta_t &=&
\Bigg\{\int_0^1\mathbb{E}'\Big(f_{\tilde{x}}(\lambda)(\tilde{X}^\theta_t)'+
f_{x}(\lambda)\tilde{X}^\theta_t+
f_{\tilde{y}}(\lambda)(\tilde{Y}^\theta_t)'+
f_{y}(\lambda)\tilde{Y}^\theta_t
+f_{\tilde{z}}(\lambda)(\tilde{Z}^\theta_t)'+
f_{z}(\lambda)\tilde{Z}^\theta_t\Big)d\lambda \nonumber\\
&&+A_t+B_t+C_t+G_t\Bigg\}dt-\tilde{Z}^\theta_tdW_t, \label{lem 4.1
(1)}
\end{eqnarray}
where we denote
\begin{eqnarray*}
A_t&=&\int_0^1\mathbb{E}'\left\{\big(f_{\tilde{x}}(\lambda)-\bar{f}_{\tilde{x}}(t,u)\big)(k_t)'+\big(f_{x}(\lambda)-\bar{f}_{x}(t,u)\big)k_t\right\}d\lambda,
\\
B_t&=&\int_0^1\mathbb{E}'\left\{\big(f_{\tilde{y}}(\lambda)-\bar{f}_{\tilde{y}}(t,u)\big)(m_t)'+\big(f_{y}(\lambda)-\bar{f}_{y}(t,u)\big)m_t\right\}d\lambda,
\\
C_t&=&\int_0^1\mathbb{E}'\left\{\big(f_{\tilde{z}}(\lambda)-\bar{f}_{\tilde{z}}(t,u)\big)(n_t)'+\big(f_{z}(\lambda)-\bar{f}_{z}(t,u)\big)n_t\right\}d\lambda,
\\
G_t&=&\int_0^1\mathbb{E}'\left\{\big(f_{v}(\lambda)-\bar{f}_{v}(t,u)\big)(v_t-u_t)\right\}d\lambda.
\end{eqnarray*}

$A_t$ tends to 0 in $L^2(\Omega \times [0,T])$ as $\theta\rightarrow
0$. Indeed, since the Lipschitz continuity of
$f_{\tilde{x}}(\tilde{x},\tilde{y},\tilde{z},x,y,z),\\
f_x(\tilde{x},\tilde{y},\tilde{z},x,y,z)$ with respect to
$(\tilde{x},\tilde{y},\tilde{z},x,y,z)$, there exists a positive
constant $C$, which may differ from line to line if not specified,
such that:
\begin{eqnarray*}
|f_{\tilde{x}}(\lambda)-\bar{f}_{\tilde{x}}(t,u)|  \leq
C\lambda\theta\beta_t,\ \ \ \ \ \ |f_{x}(\lambda)-\bar{f}_{x}(t,u)|
\leq C\lambda \theta \beta_t,
\end{eqnarray*}
 with
\begin{equation*}
\beta_t=|(\tilde{X}_t^{\theta}+k_t)'|+|(\tilde{Y}_t^{\theta}+m_t)'|+|(\tilde{Z}_t^{\theta}+n_t)'|+|\tilde{X}_t^{\theta}+k_t|+|\tilde{Y}_t^{\theta}+m_t|+|\tilde{Z}_t^{\theta}+n_t|+|v_t-u_t|.
\end{equation*}
Then
\begin{eqnarray*}
|A_t|^2&=&|\int_0^1\mathbb{E}'\left\{\big(f_{\tilde{x}}(\lambda)-\bar{f}_{\tilde{x}}(t,u)\big)(k_t)'+\big(f_{x}(\lambda)-\bar{f}_{x}(t,u)\big)k_t\right\}d\lambda|^2\\
     &\leq&
     C\theta^2\Big(\mathbb{E}'[(|(k_t)'|+|k_t|)\beta_t]\Big)^2\\
&\leq&
     C\theta^2\Big(\mathbb{E}'[|(k_t)'|^2]\mathbb{E}'[\beta_t^2]+|k_t|^2\mathbb{E}'[\beta_t^2]\Big)\\
&=&
     C\theta^2\Big(\mathbb{E}[|k_t|^2]+|k_t|^2\Big)\mathbb{E}'[\beta_t^2].
\end{eqnarray*}

\noindent So
\begin{eqnarray*}
\mathbb{E}\int_0^T|A_t|^2dt  &\leq&   C\theta^2
\mathbb{E}\left(\int_0^T
\mathbb{E}[|k_t|^2]\mathbb{E}'[\beta_t^2]dt\right)\\
&\leq&   C\theta^2 \mathbb{E}\left(\int_0^T
|k_t|^4dt\right)^{\frac{1}{2}}\mathbb{E}\left(\int_0^T
\mathbb{E}'[\beta_t^4]dt\right)^{\frac{1}{2}},
\end{eqnarray*}
which converges to 0 as $\theta\rightarrow 0$ since the expected
values are finite. Similar estimations for $B_t,C_t$ and $G_t$ in
(\ref{lem 4.1 (1)}) show that these terms also converge to 0 in
$L^2(\Omega \times [0,T])$. For  simplicity, we let
$I_t=A_t+B_t+C_t+G_t$. Using Ito's formula to
$|\tilde{Y}^\theta_t|^2$ and noting that assumption (A2), we have
\begin{eqnarray*}
\mathbb{E}|\tilde{Y}^\theta_t|^2+\mathbb{E}\int_t^T|\tilde{Z}^\theta_s|^2ds
&\leq& C\mathbb{E}\int_t^T|\tilde{Y}^\theta_s||
\mathbb{E}(\tilde{X}^\theta_s) +\tilde{X}^\theta_s
+\mathbb{E}(\tilde{Y}^\theta_s) +\tilde{Y}^\theta_s
+\mathbb{E}(\tilde{Z}^\theta_s) +\tilde{Z}^\theta_s
+I_s|ds \\
&\leq & C
\mathbb{E}\int_t^T|\tilde{Y}^\theta_s|^2ds+\frac{1}{2}\mathbb{E}\int_t^T|\tilde{Z}^\theta_s|^2ds+J_{\theta},
\end{eqnarray*}
with
\begin{equation*}
J_{\theta}=\mathbb{E}\int_t^T|\tilde{X}^\theta_s|^2ds+\mathbb{E}\int_t^T|I_s|^2ds,
\end{equation*}
where $C>0$ is a constant and $J_{\theta} \rightarrow 0$ as $\theta
\rightarrow 0$ . Applying Gronwall's lemma gives the last two
results of (\ref{convergence result for control}).
\end{proof}

\begin{lemma}
Under the assumptions (A1) and (A2), for any $v(\cdot)\in \mathcal{U}$, the
following variational inequality holds:
\begin{eqnarray}
&&\mathbb{E}\int_0^T\mathbb{E}^{\prime }\Big(\bar{h}_{\tilde{x}%
}(t)(k_t)^{\prime }+\bar{h}_{x}(t)k_t+\bar{h}_{\tilde{y}}(t)(m_t)^{\prime }+%
\bar{h}_{y}(t)m_t+ \bar{h}_{\tilde{z}}(t)(n_t)^{\prime }+\bar{h}_{z}(t)n_t+%
\bar{h}_{v}(t)(v(t)-u(t))\Big)dt  \notag \\
&&\ \ \ \ \ \ +\mathbb{E}\Big(\gamma_{y}(Y^u(0))m_0\Big)+\mathbb{E}\Big(%
g_{x}(X^u_T)k_T\Big)\geq 0.  \label{variational inequality for control}
\end{eqnarray}
where we denote $h(t,(X^u_t)^{\prime u}_t)^{\prime u}_t)^{\prime
u}_t,Y^u_t,Z^u_t,u_t)$ by $\bar{h}(t).$
\end{lemma}

\begin{proof}
Since $u(\cdot)$ is an optimal control of  the problem, then
\begin{equation}
\theta^{-1}\Big[J\big(u(\cdot)+\theta
(v(\cdot)-u(\cdot))\big)-J(u(\cdot))\Big]\geq 0. \label{cost
inequality for control}
\end{equation}
From the estimate of (\ref{convergence result for control}), when
$\theta \rightarrow 0$, it follows that
\begin{eqnarray*}
&&\frac{1}{\theta}\mathbb{E}[g(X^{\theta}_T)-g(X^u_T)] \rightarrow
\mathbb{E}\big[g_{x}(X^u_T)k_T\big],
\\
&&\frac{1}{\theta}\mathbb{E}[\gamma(Y^{\theta}(0))-\gamma(Y^u(0))]
\rightarrow \mathbb{E}[\gamma_{y}(Y^u(0))m_0],
\\
&&\frac{1}{\theta}\Big\{\mathbb{E}\int_0^T\mathbb{E}'\big[h\big(t,(X^{\theta}_t)',(Y^{\theta}_t)',(Z^{\theta}_t)',X^{\theta}_t,Y^{\theta}_t,Z^{\theta}_t,u(t)+\theta
(v(t)-u(t))\big)-\bar{h}(t)\big]dt\Big\} \rightarrow
\\
&&
\mathbb{E}\int_0^T\mathbb{E}'\Big(\bar{h}_{\tilde{x}}(t)(k_t)'+\bar{h}_{\tilde{y}}(t)(m_t)'+\bar{h}_{\tilde{z}}(t)(n_t)'+\bar{h}_{x}(t)k_t+\bar{h}_{y}(t)m_t+\bar{h}_{z}(t)n_t+\bar{h}_{v}(t)(v(t)-u(t))
\Big)dt.
\end{eqnarray*}
Combining the limits above with (\ref{cost inequality for control})
and the definition of the cost functional, we derive
(\ref{variational inequality for control}) easily.
\end{proof}

\subsection{Adjoint equation and Maximum principle}

For deriving the maximum principle, we introduce the following adjoint
equation corresponding to mean-field FBSDEs (\ref{state equation for control}%
), which is a mean-field FBSDEs and whose solution is denoted by $(p(\cdot
),q(\cdot ),Q(\cdot ))$,

\begin{eqnarray}
\left\{
\begin{array}{ll}
-dp_t=\mathbb{E}^{\prime }\Big(b_{\tilde{x}}(t,(X^u_t)^{\prime
u}_t,u_t)(p_t)^{\prime }+b_{x}(t,(X^u_t)^{\prime u}_t,u_t)p_t &  \\
\ \ \ \ \ \ \ \ \ \ +\sigma_{\tilde{x}}(t,(X^u_t)^{\prime
u}_t,u_t)(q_t)^{\prime }+\sigma_{x}(t,(X^u_t)^{\prime u}_t,u_t)q_t\Big)dt &
\\
\ \ \ \ \ \ \ \ \ \ +\mathbb{E}^{\prime }\Big(\bar{h}_{\tilde{x}}(t) +\bar{h}%
_{x}(t)-\bar{f}_{\tilde{x}}(t,u)(Q_t)^{\prime }-\bar{f}_{x}(t,u)Q_t\Big)%
dt-q_tdW_t, &  \\
dQ_t=\mathbb{E}^{\prime }\Big(\bar{f}_{\tilde{y}}(t,u)(Q_t)^{\prime }+\bar{f}%
_{y}(t,u)Q_t-\bar{h}_{\tilde{y}}(t) -\bar{h}_{y}(t)\Big)dt &  \\
\ \ \ \ \ \ \ \ \ + \mathbb{E}^{\prime }\Big(\bar{f}_{\tilde{z}%
}(t,u)(Q_t)^{\prime }+\bar{f}_{z}(t,u)Q_t-\bar{h}_{\tilde{z}}(t) -\bar{h}%
_{z}(t)\Big)dW_t, &  \\
p_{T}=g_{x}(X^u_T)-\Phi_{x}(X^u_T)Q_T,\ Q_0=-\gamma_{y}(Y^u(0)). \label%
{adjoint equation for control} &
\end{array}
\right.
\end{eqnarray}

This equation reduces to the standard one, when the coefficients do not
depend explicitly on $\omega ^{\prime }$ of the underlying diffusion. Under
the assumptions (A1) and (A2), this is a linear mean-field FBSDEs with
bounded coefficients. Moreover, due to Lemma 2.2 and Theorem 4.1 in \cite%
{MFBSDE2}, it admits a unique $\mathbb{F}$-adapted solution $(Q,p,q)$ such
that
\begin{equation*}
\mathbb{E}\Big[\sup\limits_{0\leq t\leq T}|Q(t)|^{2}+\sup\limits_{0\leq
t\leq T}|p(t)|^{2}+\int_{0}^{T}|q(t)|^{2}dt\Big]<+\infty .
\end{equation*}

\noindent Next, we define the Hamiltonian function as follows:
\begin{eqnarray*}
H(t,\tilde{x},\tilde{y},\tilde{z},x,y,z,p,q,Q,v)&=&b(t,\tilde{x}%
,x,v)p+\sigma(t,\tilde{x},x,v)q -f(t,\tilde{x},\tilde{y},\tilde{z},x,y,z,v)Q
\\
&&+h(t,\tilde{x},\tilde{y},\tilde{z},x,y,z,v).
\end{eqnarray*}

The following theorem constitutes the main result of this section.

\begin{theorem}
(SMP in Integral Form). Suppose (A1)-(A2) hold. Let $u(\cdot)$ be an optimal
control of the problem, and $(X^u(\cdot),Y^u(\cdot),Z^u(\cdot))$ denote the
corresponding trajectory. Then, for all $v \in \mathcal{U}$, there holds
\begin{eqnarray}
\mathbb{E}\int_0^T\mathbb{E}^{\prime }\big[H_{v}\big(t,X^{\prime
}_t,Y^{\prime }_t,Z^{\prime }_t,X_t,Y_t,Z_t,p_t,q_t,Q_t,u(t)\big)%
(v(t)-u(t))]dt\geq 0,  \label{SMP in control}
\end{eqnarray}
a.e.,a.s., where $(p(\cdot),q(\cdot),Q(\cdot))$ is the the solution of
adjoint equation (\ref{adjoint equation for control}).
\end{theorem}

\begin{proof}
Applying It$\hat{o}$'s formula to $k_tp_t+m_tQ_t$ yields
\begin{eqnarray*}
\lefteqn{\mathbb{E}\Big[k_Tp_T+m_TQ_T-k_0p_0-m_0Q_0\Big]}  \\
&=&\mathbb{E}\Big[g_x(X^u_T)k_T+m_0\gamma_y(Y^u(0))\Big]\\
&=&\mathbb{E}\int_0^T\mathbb{E}'\Big[\left(p_tb_v(t,(X^u_t)',X^u_t,u_t)+q_t\sigma_v(t,(X^u_t)',X^u_t,u_t)-Q_t\bar{f}_v(t,u)\right)(v(t)-u(t))\Big]dt\\
&&-\mathbb{E}\int_0^T\mathbb{E}'\Big(k_t\bar{h}_{\tilde{x}}(t)+k_t\bar{h}_{x}(t)+m_t\bar{h}_{\tilde{y}}(t)+m_t\bar{h}_{y}(t)
+n_t\bar{h}_{\tilde{z}}(t)+n_t\bar{h}_{z}(t)\Big)dt.
\end{eqnarray*}
Together with Lemma 13, we derive
\begin{eqnarray*}
&& \mathbb{E}\int_0^T\mathbb{E}'\Big[\left(p_tb_v(t,(X^u_t)',X^u_t,u_t)+q_t\sigma_v(t,(X^u_t)',X^u_t,u_t)-Q_t\bar{f}_v(t,u)+\bar{h}_{v}(t)\right)(v(t)-u(t))\Big]dt \\
&=&\mathbb{E}\int_0^T\mathbb{E}'[H_v(t,(X^u_t)',(Y^u_t)',(Z^u_t)',X^u_t,Y^u_t,Z^u_t,p_t,q_t,Q_t,u_t)(v(t)-u(t))]dt
\geq 0.
\end{eqnarray*}
Thus, we come to the conclusion of this theorem.
\end{proof}

\begin{remark}
From (\ref{SMP in control}), we can get that
\begin{equation}
\mathbb{E}^{\prime }\big[H_{v}\big(t,X^{\prime }_t,Y^{\prime }_t,Z^{\prime
}_t,X_t,Y_t,Z_t,p_t,q_t,Q_t,u(t)\big)(v-u(t))]\geq 0,
\end{equation}
$dtdP$-a.e., for any $v \in \mathcal{U}$.
\end{remark}

\subsection{Sufficient conditions for maximum principle}

This section is devoted to establish the sufficient maximum principle (also
called verification theorem) of the mean-field stochastic control problem.

We need the following additional assumptions.

\begin{enumerate}
\item[\textbf{(A3)}] The function $\Phi $ is convex in $x$. $g$ is convex in
$x$ and $\gamma $ is convex in $y$.
\end{enumerate}

\begin{theorem}
(Sufficient Conditions for the Optimality of the control) Assume that the
conditions (A1)-(A3) are satisfied. Let $u(\cdot )\in \mathcal{U}$ with
state trajectory $(X_{t}^{u},Y_{t}^{u},Z_{t}^{u})$ and $(p(\cdot ),q(\cdot
),Q(\cdot ))$ be the solution of Mean-field FBSDE (\ref{adjoint equation for
control}). Suppose
\begin{eqnarray*}
\lefteqn{\mathbb{E}\prime \lbrack H(t,(X_{t}^{u})\prime ,(Y_{t}^{u})\prime
,(Z_{t}^{u})\prime ,X_{t}^{u},Y_{t}^{u},Z_{t}^{u},p_{t},q_{t},Q_{t},u(t))]}
\\
&=&\min\limits_{v\in U}\mathbb{E}_{t}^{\prime u})_{t}^{\prime
u})_{t}^{\prime u})_{t}^{\prime u},Y_{t}^{u},Z_{t}^{u},p_{t},q_{t},Q_{t},v)]
\end{eqnarray*}%
hold for all $t\in \lbrack 0,T]$. Moreover, suppose function $H(t,\tilde{x},%
\tilde{y},\tilde{z},x,y,z,p,q,Q,v)$ is convex with respect to $(t,\tilde{x},%
\tilde{y},\tilde{z},x,y,z,v)$. Then $u$ is an optimal control of problem (%
\ref{state equation for control})-(\ref{optimal control}).
\end{theorem}

\begin{proof}
For any $v(\cdot) \in \mathcal{U}$, we consider
\begin{equation}
J(v(\cdot))-J(u(\cdot))=\textrm{I}+\textrm{II}+\textrm{III}
\label{sufficicent equality 1}
\end{equation}
with
\begin{eqnarray*}
 \textrm{I}&=&\mathbb{E}\int_0^T\mathbb{E}'\bigg[h(t,(X^v_t)',(Y^v_t)',(Z^v_t)',X^v_t,Y^v_t,Z^v_t,v(t))
 \\&& -h(t,(X^u_t)',(Y^u_t)',(Z^u_t)',X^u_t,Y^u_t,Z^u_t,u(t))
 \bigg]dt,
\\
\textrm{II}&= &\mathbb{E}\bigg[g(X^v_T)-g(X^u_T)\bigg],\\
\textrm{III}&=&\mathbb{E}\bigg[\gamma(Y^v(0))-\gamma(Y^{u}(0))\bigg].
\end{eqnarray*}
Since $g$ is convex, it holds that
\begin{eqnarray}
\textrm{II}=\mathbb{E}\big(g(X^v_T)-g(X^u_T)\big)  \geq
\mathbb{E}[g_{x}(X^u_T)(X^v_T-X^u_T)].\label{II}
\end{eqnarray}
Due to $\gamma$ is convex on $y$, we have
\begin{equation}
\textrm{III} \geq
\gamma_{y}(Y^u(0))(Y^v(0)-Y^u(0))=-Q_0(Y^v(0)-Y^u(0)).\label{III}
\end{equation}
From (\ref{II}) and (\ref{III}), we get
\begin{eqnarray}
\textrm{II}+\textrm{III} \geq
\mathbb{E}\big[g_x(X^u_T)(X^v_T-X^u_T)+\gamma_y(Y^u(0))(Y^v(0)-Y^{u}(0))]
\label{II+III}.
\end{eqnarray}

\noindent Since $\Phi$ is convex,
\begin{eqnarray}
Y^v_T-Y^{u}_T =\Phi(X^v_T)-\Phi(X^u_T) \geq
\Phi_x(X^u_T)(X^v_T-X^u_T). \label{convex}
\end{eqnarray}
By applying It\textrm{$\hat{o}$}'s formula to
$Q_t(Y^v_t-Y^u_t)+p_t(X^v_t-X^u_t)$, we get
\begin{eqnarray}
&& \mathbb{E}\big[Q_T(Y^v_T-Y^{u}_T)-Q_0(Y^v(0)-Y^{u}(0))+p_T(X^v_T-X^u_T)-p(0)(X^v(0)-X^u(0))\big] \nonumber \\
&=&\mathbb{E}\big[Q_T(Y^v_T-Y^{u}_T)+\gamma_y(Y^u(0))(Y^v(0)-Y^{u}(0))+(g_x(X^u_T)-\Phi_x(X^u_T)Q_T)(X^v_T-X^u_T)\big]  \nonumber  \\
&=&\mathbb{E}\int_0^T\mathbb{E}'\Big\{p_t(b(t,(X^v_t)',X^v_t,v(t))-b(t,(X^u_t)',X^u_t,u(t)))-Q_t(\bar{f}(t,v)-\bar{f}(t,u))\nonumber \\
&&\ \ \ \ \ \ \ \ \ \ \ \ \
+q_t(\sigma(t,(X^v_t)',X^v_t,v(t))-\sigma(t,(X^u_t)',X^u_t,u(t)))\Big\}dt+\textrm{IV},
\label{sufficient equality 2}
\end{eqnarray}
with
\begin{eqnarray*}
\textrm{IV}&=&\mathbb{E}\int_0^T\mathbb{E}'\Bigg\{\Big(\bar{f}_{\tilde{y}}(t,u)(Q_t)'+\bar{f}_{y}(t,u)Q_t-\bar{h}_{\tilde{y}}(t)-\bar{h}_{y}(t)\Big)(Y^v_t-Y^u_t)\\
&&\ \ \ \ \ \ \ \ \ \ +\Big(\bar{f}_{\tilde{z}}(t,u)(Q_t)'+\bar{f}_{z}(t,u)Q_t-\bar{h}_{\tilde{z}}(t)-\bar{h}_{z}(t)\Big)(Z^v_t-Z^u_t)\\
&&\ \ \ \ \ \ \ \ \ \
-\Big(b_{\tilde{x}}(t,(X^u_t)',X^u_t,u_t)(p_t)'+b_{x}(t,(X^u_t)',X^u_t,u_t)p_t\\
&&\ \ \ \ \ \ \ \ \ \
+\sigma_{\tilde{x}}(t,(X^u_t)',X^u_t,u_t)(q_t)'+\sigma_{x}(t,(X^u_t)',X^u_t,u_t)q_t\Big)(X^v_t-X^u_t)\\
&&\ \ \ \ \ \ \ \ \ \ -\Big(\bar{h}_{\tilde{x}}(t)
+\bar{h}_{x}(t)-\bar{f}_{\tilde{x}}(t,u)(Q_t)'-\bar{f}_{x}(t,u)Q_t\Big)(X^v_t-X^u_t)
\Bigg\}dt
\end{eqnarray*}
\begin{eqnarray*}
&=&-\mathbb{E}\int_0^T\mathbb{E}'\Big\{H_{\tilde{x}}(t,(X^u_t)',(Y^u_t)',(Z^u_t)',X^u_t,Y^u_t,Z^u_t,p_t,q_t,Q_t,u(t))(X_t^v-X^u_t)'
\\
&&\ \ \ \ \ \ \ \ \ +H_{x}(t,(X^u_t)',(Y^u_t)',(Z^u_t)',X^u_t,Y^u_t,Z^u_t,p_t,q_t,Q_t,u(t))(X_t^v-X^u_t)\\
&&\ \ \ \ \ \ \ \ \ +H_{\tilde{y}}(t,(X^u_t)',(Y^u_t)',(Z^u_t)',X^u_t,Y^u_t,Z^u_t,p_t,q_t,Q_t,u(t))(Y_t^v-Y^u_t)'\\
&&\ \ \ \ \ \ \ \ \ +H_{y}(t,(X^u_t)',(Y^u_t)',(Z^u_t)',X^u_t,Y^u_t,Z^u_t,p_t,q_t,Q_t,u(t))(Y_t^v-Y^u_t)\\
&&\ \ \ \ \ \ \ \ \ +H_{\tilde{z}}(t,(X^u_t)',(Y^u_t)',(Z^u_t)',X^u_t,Y^u_t,Z^u_t,p_t,q_t,Q_t,u(t))(Z_t^v-Z^u_t)'\\
&&\ \ \ \ \ \ \ \ \
+H_{z}(t,(X^u_t)',(Y^u_t)',(Z^u_t)',X^u_t,Y^u_t,Z^u_t,p_t,q_t,Q_t,u(t))(Z_t^v-Z^u_t)\Big\}dt.
\end{eqnarray*}

\noindent Together with (\ref{sufficicent equality 1}),
(\ref{II+III}),(\ref{convex}) and (\ref{sufficient equality 2}), we
get
\begin{eqnarray}
&&J(v(\cdot))-J(u(\cdot))=\textrm{I}+\textrm{II}+\textrm{III}
 \nonumber \\&\geq&
\textrm{I}+\mathbb{E}\int_0^T\mathbb{E}'
\Big\{p_t(b(t,(X^v_t)',X^v_t,v(t))-b(t,(X^u_t)',X^u_t,u(t)))-Q_t(\bar{f}(t,v)-\bar{f}(t,u))\nonumber \\
&&\ \ \ \ \ \ \ \ \ \ \ \ \
+q_t(\sigma(t,(X^v_t)',X^v_t,v(t))-\sigma(t,(X^u_t)',X^u_t,u(t)))\Big\}dt+\textrm{IV} \nonumber  \\
&=&\mathbb{E}\int_0^T\mathbb{E}'\bigg[H(t,(X^v_t)',(Y^v_t)',(Z^v_t)',X^v_t,Y^v_t,Z^v_t,p_t,q_t,Q_t,v(t))
 \nonumber \\&& \ \ \ \  \ \ \ \ \  -H(t,(X^u_t)',(Y^u_t)',(Z^u_t)',X^u_t,Y^u_t,Z^u_t,p_t,q_t,Q_t,u(t))
 \bigg]dt+\textrm{IV}. \label{sufficient equality 3}
\end{eqnarray}

\noindent Noticing $H$ is convex with respect to
$(\tilde{x},\tilde{y},\tilde{z},x,y,z)$, the use of the Clark
generalized gradient of $H$, evaluated at
$((X_t^u)',(Y_t^u)',(Z_t^u)',X_t^u,Y_t^u,Z_t^u)$, yields
\begin{eqnarray*}
&&H(t,(X^v_t)',(Y^v_t)',(Z^v_t)',X^v_t,Y^v_t,Z^v_t,p_t,q_t,Q_t,v(t))\\&&-H(t,(X^u_t)',(Y^u_t)',(Z^u_t)',X^u_t,Y^u_t,Z^u_t,p_t,q_t,Q_t,u(t))\\
& \geq
&H_{\tilde{x}}(t,(X^u_t)',(Y^u_t)',(Z^u_t)',X^u_t,Y^u_t,Z^u_t,p_t,q_t,Q_t,u(t))(X_t^v-X^u_t)'
\\&&+H_{x}(t,(X^u_t)',(Y^u_t)',(Z^u_t)',X^u_t,Y^u_t,Z^u_t,p_t,q_t,Q_t,u(t))(X_t^v-X^u_t)\\
&&+H_{\tilde{y}}(t,(X^u_t)',(Y^u_t)',(Z^u_t)',X^u_t,Y^u_t,Z^u_t,p_t,q_t,Q_t,u(t))(Y_t^v-Y^u_t)'\\
&&+H_{y}(t,(X^u_t)',(Y^u_t)',(Z^u_t)',X^u_t,Y^u_t,Z^u_t,p_t,q_t,Q_t,u(t))(Y_t^v-Y^u_t)\\
&&+H_{\tilde{z}}(t,(X^u_t)',(Y^u_t)',(Z^u_t)',X^u_t,Y^u_t,Z^u_t,p_t,q_t,Q_t,u(t))(Z_t^v-Z^u_t)'\\
&&+H_{z}(t,(X^u_t)',(Y^u_t)',(Z^u_t)',X^u_t,Y^u_t,Z^u_t,p_t,q_t,Q_t,u(t))(Z_t^v-Z^u_t)\\
&&+H_{v}(t,(X^u_t)',(Y^u_t)',(Z^u_t)',X^u_t,Y^u_t,Z^u_t,p_t,q_t,Q_t,u(t))(v_t-u_t).
\end{eqnarray*}
Combined the above inequality with (\ref{sufficient equality 3}), we
have
\begin{eqnarray*}
\lefteqn{J(v(\cdot))-J(u(\cdot))\geq } \nonumber \\
&&\mathbb{E}\int_0^T\mathbb{E}'\Big[H_{v}(t,(X^u_t)',(Y^u_t)',(Z^u_t)',X^u_t,Y^u_t,Z^u_t,p_t,q_t,Q_t,u(t))\big(v(t)-u(t)\big)\Big]dt\geq 0. \nonumber\\
 \label{sufficient equality 4}
\end{eqnarray*}
Hence, we draw the desired conclusion.
\end{proof}

\section{Stochastic maximum principle for fully coupled forward-backward
stochastic control systems of mean-field type}

In this section, we extend control problems to the fully coupled mean-field
FBSDEs. For any $v(\cdot)\in \mathcal{U}$, the state equation consists of
the following forward-backward control system of mean-field type:
\begin{eqnarray}
\left\{
\begin{array}{ll}
dX_t=\mathbb{E}^{\prime }[b(t,X^{\prime }_t,Y^{\prime }_t,Z^{\prime
}_t,X_t,Y_t,Z_t, v_t)]dt+\mathbb{E}^{\prime }[\sigma(t,X^{\prime
}_t,Y^{\prime }_t,Z^{\prime }_t,X_t,Y_t,Z_t, v_t)]dW_t, &  \\
X(0)=x_0, &  \\
-dY_t=\mathbb{E}^{\prime }[f(t,X^{\prime }_t,Y^{\prime }_t,Z^{\prime
}_t,X_t,Y_t,Z_t,v_t)]dt-Z_tdW_t, \label{state equation for coupled control}
&  \\
Y_T=\Phi(X_T), &
\end{array}
\right.
\end{eqnarray}
where
\begin{eqnarray*}
\begin{array}{ll}
b: [0,T]\times \mathbb{R} \times \mathbb{R}\times \mathbb{R}^d \times
\mathbb{R}\times \mathbb{R}\times \mathbb{R}^d \times U \rightarrow \mathbb{R%
}, &  \\
\sigma: [0,T]\times \mathbb{R} \times \mathbb{R}\times \mathbb{R}^d \times
\mathbb{R}\times \mathbb{R}\times \mathbb{R}^d \times U \rightarrow \mathbb{R%
}^d, &  \\
f: [0,T]\times \mathbb{R} \times \mathbb{R}\times \mathbb{R}^d \times
\mathbb{R}\times \mathbb{R}\times \mathbb{R}^d \times U \rightarrow \mathbb{R%
}, &  \\
\Phi:\mathbb{R} \rightarrow \mathbb{R}. &
\end{array}%
\end{eqnarray*}

\noindent The expected cost function is given by:
\begin{eqnarray}
J(v(\cdot))&=&\mathbb{E}\bigg(\int_0^T\mathbb{E}^{\prime }[h(t,X^{\prime
}_t,Y^{\prime }_t,Z^{\prime }_t,X_t,Y_t,Z_t,v(t))]dt+g(X_T)+\gamma(Y(0))%
\bigg),  \label{cost functional for coupled control}
\end{eqnarray}
where
\begin{eqnarray*}
\begin{array}{ll}
g:\mathbb{R}\rightarrow \mathbb{R}, &  \\
\gamma:\mathbb{R}\rightarrow \mathbb{R}, &  \\
h: [0,T]\times \mathbb{R} \times \mathbb{R}\times \mathbb{R}^d\times \mathbb{%
R}\times \mathbb{R}\times \mathbb{R}^d \times U \rightarrow \mathbb{R}. &
\end{array}%
\end{eqnarray*}

The optimal control problem is to minimize the functional $J(\cdot)$ over $U$%
. A control that solves this problem is called optimal.

\noindent We assume:

\begin{enumerate}
\item[\textbf{(A4)}]
\begin{equation*}
\left\{
\begin{array}{ll}
\text{(i)}\ \ \ b,\sigma ,f,\Phi ,h,g\ \text{and}\ \gamma \
\mbox{are continuously
differentiable}; &  \\
\text{(ii)}\ \ \mbox{The derivatives of}\ b,\sigma ,f\ \mbox{and}\ \Phi \ %
\mbox{are bounded}; &  \\
\text{(iii)}\ \mbox{The derivatives of}\ h\ \mbox{are bounded by}\ C(1+|%
\tilde{x}|+|\tilde{y}|+|\tilde{z}|+|x|+|y|+|z|); &  \\
\text{(iv)}\ \mbox{The derivatives of}\ g\ \mbox{and}\gamma \
\mbox{with
respect to}\ x\ \mbox{and}\ y\ \mbox{are bounded by}\ C(1+|x|)\  &  \\
\ \ \ \ \ \ \mbox{and}\ C(1+|y|)\ \mbox{respectively}; &  \\
\text{(v)}\ \ \mbox{For any given admissible control}\ v(\cdot ),\
\mbox{the equation (\ref{state equation for coupled control})
satisfies (H4) and (H5)}. &
\end{array}%
\right.
\end{equation*}
\end{enumerate}

According to Theorem 3.1, for any given admissible control $v(\cdot )\in
\mathcal{U}$, there exists a unique adapted solution $%
(X_{t}^{v},Y_{t}^{v},Z_{t}^{v})$ satisfying the fully coupled mean-field
FBSDEs (\ref{state equation for coupled control}).

Let $u(\cdot)$ be an optimal control and $(X^u(\cdot),Y^u(\cdot),Z^u(\cdot))$
be the corresponding state trajectory of stochastic control system. In this
case, the corresponding adjoint equation becomes
\begin{eqnarray}
\left\{
\begin{array}{ll}
-dp(t)=\mathbb{E}^{\prime }\Big(\bar{b}_{\tilde{x}}(t)(p_t)^{\prime }+\bar{b}%
_{x}(t)p_t+\bar{\sigma}_{\tilde{x}}(t)(q_t)^{\prime }+\bar{\sigma}_{x}(t)q_t%
\Big)dt &  \\
\ \ \ \ \ \ \ \ \ \ \ \ \ \ +\mathbb{E}^{\prime }\Big(\bar{h}_{\tilde{x}}(t)
+\bar{h}_{x}(t)-\bar{f}_{\tilde{x}}(t)(Q_t)^{\prime }-\bar{f}_{x}(t)Q_t\Big)%
dt-q_tdW_t, &  \\
dQ_t=\mathbb{E}^{\prime }\Big(\bar{f}_{\tilde{y}}(t)(Q_t)^{\prime }+\bar{f}%
_{y}(t)Q_t-\bar{b}_{\tilde{y}}(t)(p_t)^{\prime }-\bar{b}_{y}(t)p_t-\bar{%
\sigma}_{\tilde{y}}(t)(q_t)^{\prime }-\bar{\sigma}_{y}(t)q_t -\bar{h}_{%
\tilde{y}}(t)-\bar{h}_{y}(t)\Big)dt &  \\
\ \ \ \ \ \ \ \ \ + \mathbb{E}^{\prime }\Big(\bar{f}_{\tilde{z}%
}(t)(Q_t)^{\prime }+\bar{f}_{z}(t)Q_t-\bar{b}_{\tilde{z}}(t)(p_t)^{\prime }-%
\bar{b}_{z}(t)p_t-\bar{\sigma}_{\tilde{z}}(t)(q_t)^{\prime }-\bar{\sigma}%
_{z}(t)q_t -\bar{h}_{\tilde{z}}(t)-\bar{h}_{z}(t)\Big)dW_t, &  \\
p_{T}=g_{x}(X^u_T)-\Phi_x(X^u_T)Q_T,\ Q_0=-\gamma_{y}(Y^u(0)), \label%
{adjoint equation for coupled control} &
\end{array}
\right.
\end{eqnarray}
in which we use the notation $\bar{\psi}(t)=\psi(t,(X^u_t)^{\prime
u}_t)^{\prime u}_t)^{\prime u}_t, Y^u_t,Z^u_t)$ for $\psi=b,\sigma,f,h$.
When the coefficients $b, \sigma$ and $f$ do not depend explicitly on $%
\omega^{\prime }$, the adjoint equation (\ref{adjoint equation for coupled
control}) reduces to the standard adjoint equation (see Shi and Wu \cite{Shi
Wu 2006}) corresponding to fully coupled FBSDE.

On the other hand, from the assumption (A4) and the fact that (\ref{state
equation for coupled control}) satisfies (H4) and (H5), we can easily verify
that this adjoint equation (\ref{adjoint equation for coupled control})
satisfies (H4) and (H6). Then, from Theorem 3.2, we know that (\ref{adjoint
equation for coupled control}) has a unique $\mathbb{F}$-adapted solution $%
(Q,p,q)$ such that
\begin{equation*}
\mathbb{E}\Big[\sup\limits_{0\leq t\leq T}|Q(t)|^{2}+\sup\limits_{0\leq
t\leq T}|p(t)|^{2}+\int_{0}^{T}|q(t)|^{2}dt\Big]<+\infty .
\end{equation*}

\noindent Define the Hamiltonian function as
\begin{eqnarray}
H(t,\tilde{x},\tilde{y},\tilde{z},x,y,z,p,q,Q,v)&=&b(t,\tilde{x},\tilde{y},%
\tilde{z},x,y,z,v)p+\sigma(t,\tilde{x},\tilde{y},\tilde{z},x,y,z,v)q  \notag
\\
&&-f(t,\tilde{x},\tilde{y},\tilde{z},x,y,z,v)Q+h(t,\tilde{x},\tilde{y},%
\tilde{z},x,y,z,v).  \label{Hamiltonian for coupled control}
\end{eqnarray}

The following two theorems, whose proof are similar to Theorem 14 and
Theorem 16, respectively and thus are omited, are the main contribution of
this section.

\begin{theorem}
(SMP in Integral Form). Under assumptions (A4), if $u(\cdot)$ is an optimal
control with state trajectory $(X^u(\cdot),Y^u(\cdot),Z^u(\cdot))$, then
there exists a pair $(p(\cdot),q(\cdot),Q(\cdot))$ of adapted processes
which satisfies (\ref{adjoint equation for coupled control}), such that
\begin{eqnarray}
\mathbb{E}\int_0^T\mathbb{E}^{\prime }\big[H_{v}\big(t,X^{\prime
}_t,Y^{\prime }_t,Z^{\prime }_t,X_t,Y_t,Z_t,p_t,q_t,Q_t,u(t)\big)%
(v(t)-u(t))]dt\geq 0,
\end{eqnarray}
$\mathbb{P}$-a.s., for all $t \in [0,T].$
\end{theorem}

\begin{theorem}
(Sufficient Conditions for the Optimality of the Control) Assume the
condition (A4) is satisfied and let $u(\cdot)\in \mathcal{U}$ with state
trajectory $(X^u_t,Y^u_t,Z^u_t)$ be given such that there exist solutions $%
(p(\cdot),q(\cdot),Q(\cdot))$ to the adjoint equation (\ref{adjoint equation
for coupled control}). Moreover, suppose the functions $g,\gamma,\Phi $ are
convex and $H(t,\tilde{x},\tilde{y},\tilde{z},x,y,z,p,q,Q,v)$ is convex in $(%
\tilde{x},\tilde{y},\tilde{z},x,y,z,v)$. Then, if
\begin{eqnarray*}
\lefteqn{\mathbb{E}%
'[H(t,(X^u_t)',(Y^u_t)',(Z^u_t)',X^u_t,Y^u_t,Z^u_t,p_t,q_t,Q_t,u(t))]} \\
&&=\min\limits_{v \in U}\mathbb{E}^{\prime u}_t)^{\prime u}_t)^{\prime
u}_t)^{\prime u}_t,Y^u_t,Z^u_t,p_t,q_t,Q_t,v)],
\end{eqnarray*}
for all $t\in [0,T]$, $\mathbb{P}$-a.s., $u$ is an optimal control of
problem (\ref{state equation for coupled control})-(\ref{cost functional for
coupled control}).
\end{theorem}

\section{Maximum principle for mean-field stochastic games of FBSDEs}

In this section, we consider a class of non-zero sum differential games
where state variables are described by the system of FBSDEs of mean-field
type. Our objective is to derive necessary conditions for optimality in the
form of a stochastic maximum principle and the corresponding verification
theorem.

We always use the subscript 1 (respectively, subscript 2) to characterize
the variables corresponding to Player 1 (respectively, Player 2).

Let action space $U_i$ be a non-empty, closed and convex subset of $\mathbb{R%
}^k\ (i=1,2, k \in \mathbb{N}^{+})$. The admissible control set is defined
as
\begin{equation*}
\mathcal{U}_i=\{v_i \in L^2_{\mathcal{F}}(0,T;\mathbb{R}^{k})|v_i \in U_i,
t\in [0,T]\} \ (i=1,2).
\end{equation*}
For any $v_i(\cdot)\in \mathcal{U}_i\ (i=1,2)$, we consider the following
mean-field FBSDE:
\begin{eqnarray}
\left\{
\begin{array}{ll}
dX_t=\mathbb{E}^{\prime }[b(t,X^{\prime }_t, X_t,v_1(t),v_2(t))]dt+\mathbb{E}%
^{\prime }[\sigma(t,X^{\prime }_t,X_t,v_1(t),v_2(t))]dW_t, &  \\
X(0)=x_0, &  \\
-dY_t=\mathbb{E}^{\prime }[f(t,X^{\prime }_t,Y^{\prime }_t,Z^{\prime
}_t,X_t,Y_t,Z_t,v_1(t),v_2(t))]dt-Z_tdW_t, \label{state equation for game} &
\\
Y_T=\Phi(X_T), &
\end{array}
\right.
\end{eqnarray}
where
\begin{eqnarray*}
\begin{array}{ll}
b:[0,T]\times \mathbb{R}\times\mathbb{R} \times U_1 \times U_2 \rightarrow
\mathbb{R}, &  \\
\sigma:[0,T]\times \mathbb{R}\times\mathbb{R} \times U_1 \times U_2
\rightarrow \mathbb{R}^d, &  \\
f: [0,T]\times \mathbb{R} \times \mathbb{R}\times \mathbb{R}^d\times \mathbb{%
R}\times \mathbb{R}\times \mathbb{R}^d \times U_1 \times U_2 \rightarrow
\mathbb{R}, &  \\
\Phi:\mathbb{R} \rightarrow \mathbb{R}. &
\end{array}%
\end{eqnarray*}

\noindent Ensuring to achieve the goal $\Phi(x_T)$, Player $i \ (i=1,2)$,
who has his own benefits, aims at minimizing the following expected cost
functionals:
\begin{eqnarray}
J_i(v_1(\cdot),v_2(\cdot))&=&\mathbb{E}\int_0^T\mathbb{E}^{\prime }\big[%
h_i(t,X^{\prime }_t,Y^{\prime }_t,Z^{\prime }_t,X_t,Y_t,Z_t,v_1(t),v_2(t))%
\big]dt  \notag \\
&&+\mathbb{E}\Big(g_i(X_T)+\gamma_i(Y(0))\Big),
\end{eqnarray}
where
\begin{eqnarray*}
\begin{array}{ll}
g_i: \mathbb{R}\rightarrow \mathbb{R}\ (i=1,2), &  \\
\gamma_i: \mathbb{R}\rightarrow \mathbb{R}\ (i=1,2), &  \\
h_i: [0,T]\times \mathbb{R} \times \mathbb{R}\times \mathbb{R}^d\times
\mathbb{R}\times \mathbb{R}\times \mathbb{R}^d \times U_1 \times U_2
\rightarrow \mathbb{R} \ (i=1,2). &
\end{array}%
\end{eqnarray*}

Suppose each player hopes to minimize her/his cost functional $%
J_i(v_1(\cdot),v_2(\cdot))$ by selecting an appropriate admissible control $%
v_i(\cdot)\ (i=1,2)$. The problem is then to find a pair of admissible
controls $(u_1(\cdot),u_2(\cdot)) \in \mathcal{U}_1 \times \mathcal{U}_2$,
called a Nash equilibrium point for the non-zero sum game, such that
\begin{eqnarray}
\left\{
\begin{array}{ll}
J_1(u_1(\cdot),u_2(\cdot))=\min \limits_{v_1(\cdot) \in \mathcal{U}_1}
J_1(v_1(\cdot),u_2(\cdot)) &  \\
J_2(u_1(\cdot),u_2(\cdot))=\min \limits_{v_2(\cdot) \in \mathcal{U}_2}
J_2(u_1(\cdot),v_2(\cdot)) &
\end{array}
\right.  \label{equilibrium point for game}
\end{eqnarray}
We call the problem above a forward-backward non-zero sum stochastic
differential game of mean-field type, where the word ``forward-backward"
means that the game system is described by a FBSDE and the reason for
calling ``mean-field" is the coefficients of the state equation and cost
functionals depend on the law of the state process. For simplicity, we
denote it by \itshape Problem \upshape(\itshape FBNM\upshape).

We assume that the following hypothesis holds.

\begin{enumerate}
\item[\textbf{(A5)}] (i) The given functions $b(t,\tilde{x}%
,x,v_{1},v_{2}),\sigma (t,\tilde{x},x,v_{1},v_{2}),f(t,\tilde{x},\tilde{y},%
\tilde{z},x,y,z,v_{1},v_{2}),\Phi (x),\newline
h_{i}(t,\tilde{x},\tilde{y},\tilde{z},x,y,z,v_{1},v_{2}),g_{i}(x)$ and $%
\gamma _{i}(y)\ (i=1,2)$ are continuously differentiable with respect to all
of the components in these functions.

(ii) All the derivatives in (i) are Lipschitz continuous and bounded.
\end{enumerate}

For any admissible controls $v_1(\cdot)$ and $v_2(\cdot)$, we suppose that
(A5) hold. Then we know mean-field FBSDE (\ref{state equation for game})
admits a unique solution $(x^{v_1,v_2}(\cdot),y^{v_1,v_2}(%
\cdot),z^{v_1,v_2}(\cdot))$ by Lemma 2.2 and Lemma 2.3, which is called the
corresponding trajectory.

\subsection{A Pontryagin's stochastic maximum principle}

Let $(u_1(\cdot),u_2(\cdot))$ be a Nash equilibrium point of Problem (FBNM)
and $(X(\cdot),Y(\cdot),Z(\cdot))$ be the corresponding state trajectory of
game system. For any given $v_i(\cdot) \in \mathcal{U}_i\ (i=1,2)$, since $%
\mathcal{U}_i$ is convex, then $u_i^{\theta}(\cdot)=u_i(\cdot)+\theta
(v_i(\cdot)-u_i(\cdot))\in \mathcal{U}_i \ (i=1,2),$ $\forall\ \theta \in
[0,1]$.

We introduce the short-hand notation which will be in force in this section
\begin{eqnarray*}
\bar{b}(t)&=&b(t,X^{\prime }_t,X_t,u_1(t),u_2(t)), \\
\bar{\sigma}(t)&=&\sigma(t,X^{\prime }_t,X_t,u_1(t),u_2(t)), \\
\bar{f}(t)&=& f(t,X^{\prime }_t,Y^{\prime }_t,Z^{\prime
}_t,X_t,Y_t,Z_t,u_1(t),u_2(t)), \\
\bar{h}_i(t)&=& h_i(t,X^{\prime }_t,Y^{\prime }_t,Z^{\prime
}_t,X_t,Y_t,Z_t,u_1(t),u_2(t)).
\end{eqnarray*}

\noindent Let $(k(\cdot),m(\cdot),n(\cdot))$ be the solution of the
following variational equation which is a linear mean-field FBSDE:

\begin{eqnarray*}
&&\left\{
\begin{array}{ll}
dk_t=\mathbb{E}^{\prime }\Big[\bar{b}_{\tilde{x}}(t)(k_t)^{\prime }+\bar{b}%
_x(t)k_t+\bar{b}_{v_1}(t)(v_1(t)-u_1(t))+\bar{b}_{v_2}(t)(v_2(t)-u_2(t))\Big]%
dt &  \\
\ \ \ \ \ \ \ +\mathbb{E}^{\prime }\Big[\bar{\sigma}_{\tilde{x}%
}(t)(k_t)^{\prime }+\bar{\sigma}_x(t)k_t+\bar{\sigma}%
_{v_1}(t)(v_1(t)-u_1(t))+\bar{\sigma}_{v_2}(t)(v_2(t)-u_2(t))\Big]dW_t, %
\nonumber &  \\
k_{0}=0, &
\end{array}
\right. \\
&&\left\{
\begin{array}{ll}
dm_t=-\mathbb{E}^{\prime }[\bar{f}_{\tilde{x}}(t)(k_t)^{\prime }+\bar{f}%
_{x}(t)k_t+\bar{f}_{\tilde{y}}(t)(m_t)^{\prime }+\bar{f}_{y}(t)m_t +\bar{f}_{%
\tilde{z}}(t)(n_t)^{\prime }+\bar{f}_{z}(t)n_t &  \\
\ \ \ \ \ \ \ \ \ \ \ \ \ \ \ +\bar{f}_{v_1}(t)(v_1(t)-u_1(t))+\bar{f}%
_{v_2}(t)(v_2(t)-u_2(t))]dt+n_tdW_t, &  \\
m_T=k_T\Phi_x(X_T), &
\end{array}
\right.
\end{eqnarray*}

The adjoint equation corresponding to state trajectory $(X^{v_1,v_2}(%
\cdot),Y^{v_1,v_2}(\cdot),Z^{v_1,v_2}(\cdot))$, which is a mean-field FBSDE
and whose solution is denoted by $(p^{v_1,v_2}_i(\cdot),q^{v_1,v_2}_i(%
\cdot),Q^{v_1,v_2}_i(\cdot))$, satisfies
\begin{eqnarray}
\left\{
\begin{array}{ll}
-dp^{v_1,v_2}_i(t)=\mathbb{E}^{\prime }\Big(\bar{b}_{\tilde{x}%
}(t)(p^{v_1,v_2}_i(t))^{\prime }+\bar{b}_{x}(t)p^{v_1,v_2}_i(t)+\bar{\sigma}%
_{\tilde{x}}(t)(q^{v_1,v_2}_i(t))^{\prime }+\bar{\sigma}%
_{x}(t)p^{v_1,v_2}_i(t)\Big)dt &  \\
\ \ \ \ \ \ \ \ \ \ \ \ \ \ \ \ \ \ +\mathbb{E}^{\prime }\Big(\bar{h}_{i%
\tilde{x}}(t) +\bar{h}_{ix}(t)-\bar{f}_{\tilde{x}}(t)(Q^{v_1,v_2}_i(t))^{%
\prime }-\bar{f}_{x}(t)Q^{v_1,v_2}_i(t)\Big)dt-q^{v_1,v_2}_i(t)dW_t, &  \\
dQ^{v_1,v_2}_i(t)=\mathbb{E}^{\prime }\Big(\bar{f}_{\tilde{y}%
}(t)(Q^{v_1,v_2}_i(t))^{\prime }+\bar{f}_{y}(t)Q^{v_1,v_2}_i(t)-\bar{h}_{i%
\tilde{y}}(t) -\bar{h}_{iy}(t)\Big)dt &  \\
\ \ \ \ \ \ \ \ \ \ \ \ \ \ \ \ \ \ + \mathbb{E}^{\prime }\Big(\bar{f}_{%
\tilde{z}}(t)(Q^{v_1,v_2}_i(t))^{\prime }+\bar{f}_{z}(t)Q^{v_1,v_2}_i(t)-%
\bar{h}_{i\tilde{z}}(t) -\bar{h}_{iz}(t)\Big)dW_t, &  \\
p^{v_1,v_2}_i(T)=g_{x}(X_T)-\Phi_{x}(X_T)Q^{v_1,v_2}_i(t),\
Q_0=-\gamma_{y}(Y(0)), \label{adjoint equation for game} &
\end{array}
\right.
\end{eqnarray}
where $h_{ix}$ denotes the partial derivatives of $h_i$ with respect to $x$.

By (A5) and Lemma 2.2, we can easily verify that the linear FBSDE of
mean-field type (\ref{adjoint equation for game}) admits a unique solution $%
(p^{v_1,v_2}_i(\cdot),q^{v_1,v_2}_i(\cdot),Q^{v_1,v_2}_i(\cdot))$.

\noindent The Hamiltonian function associated with random variables is
defined as follows:
\begin{eqnarray}
H_i(t,\tilde{x},\tilde{y},\tilde{z},x,y,z,p_i,q_i,Q_i,v_1,v_2)&=&p_i(t)b(t,%
\tilde{x},x,v_1,v_2)+q_i(t)\sigma(t,\tilde{x},x,v_1,v_2)  \notag \\
&&-f(t,\tilde{x},\tilde{y},\tilde{z},x,y,z,v_1,v_2)Q_i(t)  \notag \\
&&+h_i(t,\tilde{x},\tilde{y},\tilde{z},x,y,z,v_1,v_2).
\label{Hamiltonian for game}
\end{eqnarray}

Fix $u_2(\cdot)$ (respectively, $u_1(\cdot)$), to minimize the cost
functional $J_1(v_1(\cdot),u_2(\cdot))$ (respectively, $J_2(u_1(\cdot),v_2(%
\cdot))$) subject to (\ref{state equation for game}) over $\mathcal{U}_1$
(respectively, $\mathcal{U}_2$) is an optimal control problem of mean-field
FBSDEs. Following the idea developed in Section 4, it is not difficult to
analyze the game problem. Thus, we omit the detailed deduction and only
state the main result for simplicity.

\begin{theorem}
{\upshape (Stochastic Maximum Principle for SDGs)} Suppose (A5) hold. Let $%
(u_1(\cdot),u_2(\cdot))$ be a Nash equilibrium point for our stochastic game
problem (FBNM), $(X(\cdot),Y(\cdot),Z(\cdot))$ be the corresponding
trajectory and $(p^{v_1,v_2}_i(\cdot),q^{v_1,v_2}_i(\cdot),Q^{v_1,v_2}_i(%
\cdot))$ be the solution of adjoint equation (\ref{adjoint equation for game}%
). Then we have
\begin{eqnarray*}
&&\mathbb{E}\int_0^T\mathbb{E}^{\prime }\big[H_{1v_1}\big(t,X^{\prime
}_t,Y^{\prime }_t,Z^{\prime
}_t,X_t,Y_t,Z_t,p_1(t),q_1(t),Q_1(t),u_1(t),u_2(t)\big)(v_1(t)-u_1(t))]dt%
\geq 0, \\
&&\mathbb{E}\int_0^T\mathbb{E}^{\prime }\big[H_{2v_2}\big(t,X^{\prime
}_t,Y^{\prime }_t,Z^{\prime
}_t,X_t,Y_t,Z_t,p_2(t),q_2(t),Q_2(t),u_1(t),u_2(t)\big)(v_2(t)-u_2(t))]dt%
\geq 0, \\
&&\ \ \ \ \ \ \ \ \ \ \ \ \ \ \ \ \forall \ (v_1,v_2) \in U_1 \times U_2, \
\ \ \ \ a.e.t \in [0,T],\ a.s.,
\end{eqnarray*}
where the Hamiltonian function $H_i$ is defined by (\ref{Hamiltonian for
game}).
\end{theorem}

\subsection{Sufficient conditions for maximum principle}

We will establish the sufficient maximum principle (also called verification
theorem) of Problem (FBNM).

\begin{theorem}
{\upshape (Sufficient Conditions for the equilibrium point of Problem (FBNM))%
} Let (A5) hold and suppose that $(u_1(\cdot),u_2(\cdot))\in \mathcal{U}%
_1\times \mathcal{U}_2$ with state trajectory $(X_t,Y_t,Z_t)$ satisfies:
\begin{eqnarray*}
\lefteqn{\mathbb{E}%
'[H_1(t,X'_t,Y'_t,Z'_t,X_t,Y_t,Z_t,p_1^{u_1,u_2}(t),q_1^{u_1,u_2}(t),Q_1^{u_1,u_2}(t),u_1(t),u_2(t))]%
} \\
&&=\min\limits_{v_1 \in U_1}\mathbb{E}^{\prime }[H_1(t,X^{\prime
}_t,Y^{\prime }_t,Z^{\prime
}_t,X_t,Y_t,Z_t,p_1^{u_1,u_2}(t),q_1^{u_1,u_2}(t),Q_1^{u_1,u_2}(t),v_1,u_2(t))],
\\
\lefteqn{\mathbb{E}%
'[H_2(t,X'_t,Y'_t,Z'_t,X_t,Y_t,Z_t,p_2^{u_1,u_2}(t),q_2^{u_1,u_2}(t),Q_2^{u_1,u_2}(t),u_1(t),u_2(t))]%
} \\
&&=\min\limits_{v_2 \in U_2}\mathbb{E}^{\prime }[H_2(t,X^{\prime
}_t,Y^{\prime }_t,Z^{\prime
}_t,X_t,Y_t,Z_t,p_2^{u_1,u_2}(t),q_2^{u_1,u_2}(t),Q_2^{u_1,u_2}(t),u_1(t),v_2)],
\end{eqnarray*}
for all $t\in [0,T]$, where $%
p_i^{u_1,u_2}(t),q_i^{u_1,u_2}(t),Q_i^{u_1,u_2}(t)$ is the solution of
adjoint equation (\ref{adjoint equation for game}). We further assume that
the functions $\Phi(x), g_i(x), \gamma_i(y)$ and Hamiltonian function $H_i\
(i=1,2)$ are convex in $(\tilde{x},\tilde{y},\tilde{z},x,y,z,v_1,v_2)$.
Then, $(u_1(\cdot),u_2(\cdot))$ is an equilibrium point of problem (\ref%
{state equation for game})-(\ref{equilibrium point for game}).
\end{theorem}

\begin{proof}
From the proof of Theorem 16, we affirm that
\begin{equation*}
J_1(v_1(\cdot),u_2(\cdot))-J_1(u_1(\cdot),u_2(\cdot)) \geq 0
\end{equation*}
holds for any $v_1(\cdot) \in \mathcal{U}_1,$ and
\begin{equation*}
J_2(u_1(\cdot),u_2(\cdot))=\min\limits_{v_2(\cdot)\in \mathcal
{U}_2}J_2(u_1(\cdot),v_2(\cdot))
\end{equation*}
holds for any $v_2(\cdot) \in \mathcal{U}_2.$ Hence, we draw the
desired conclusion.
\end{proof}

\begin{remark}
Note that if Eq. (\ref{state equation for game}) does not include the
``forward'' part and without the influence of $\omega^{\prime }$, then the
stochastic game problem and the corresponding conclusion reduce to the case
introduced by Wang and Yu \cite{Wang Yu}.
\end{remark}

\section{Maximum principle for Mean-field stochastic games of fully coupled
FBSDEs}

In this section, we study the mean-field games of fully coupled FBSDEs. That
is, the state equation is characterized by following fully coupled FBSDEs:
\begin{eqnarray}
\left\{
\begin{array}{ll}
dX_t=\mathbb{E}^{\prime }[b(t,X^{\prime }_t,Y^{\prime }_t,Z^{\prime
}_t,X_t,Y_t,Z_t,v_1(t),v_2(t))]dt+\mathbb{E}^{\prime }[\sigma(t,X^{\prime
}_t,Y^{\prime }_t,Z^{\prime }_t,X_t,Y_t,Z_t,v_1(t),v_2(t))]dW_t, &  \\
X(0)=x_0, &  \\
-dY_t=\mathbb{E}^{\prime }[f(t,X^{\prime }_t,Y^{\prime }_t,Z^{\prime
}_t,X_t,Y_t,Z_t,v_1(t),v_2(t))]dt-Z_tdW_t, \label{state equation for coupled
game} &  \\
Y_T=\Phi(X_T), &
\end{array}
\right.
\end{eqnarray}
where
\begin{eqnarray*}
\begin{array}{ll}
b: [0,T]\times \mathbb{R} \times \mathbb{R}\times \mathbb{R}^d \times
\mathbb{R}\times \mathbb{R}\times \mathbb{R}^d\times U_1 \times U_2
\rightarrow \mathbb{R}, &  \\
\sigma: [0,T]\times \mathbb{R} \times \mathbb{R}\times \mathbb{R}^d \times
\mathbb{R}\times \mathbb{R}\times \mathbb{R}^d \times U_1 \times U_2
\rightarrow \mathbb{R}^d, &  \\
f: [0,T]\times \mathbb{R} \times \mathbb{R}\times \mathbb{R}^d \times
\mathbb{R}\times \mathbb{R}\times \mathbb{R}^d\times U_1 \times U_2
\rightarrow \mathbb{R}, &  \\
\Phi:\mathbb{R} \rightarrow \mathbb{R}. &
\end{array}%
\end{eqnarray*}
For any admissible $v_i(\cdot)\in \mathcal{U}_i\ (i=1,2)$, if conditions
(H4) and (H5) hold, the fully coupled mean-field FBSDE (\ref{state equation
for coupled game}) has a unique $\mathbb{F}$-adapted solution $%
(X^{v_1,v_2}(\cdot),Y^{v_1,v_2}(\cdot),Z^{v_1,v_2}(\cdot))$ according to
Theorem 3.1 .

For Player $i \ (i=1,2)$, the expected cost functionals is defined as
follows:
\begin{eqnarray}
J_i(v_1(\cdot),v_2(\cdot))&=&\mathbb{E}\int_0^T\mathbb{E}^{\prime }\big[%
h_i(t,X^{\prime }_t,Y^{\prime }_t,Z^{\prime }_t,X_t,Y_t,Z_t,v_1(t),v_2(t))%
\big]dt  \notag \\
&&+\mathbb{E}\Big(g_i(X_T)+\gamma_i(Y(0))\Big),
\end{eqnarray}
where
\begin{eqnarray*}
\begin{array}{ll}
g_i: \mathbb{R}\rightarrow \mathbb{R}\ (i=1,2), &  \\
\gamma_i: \mathbb{R}\rightarrow \mathbb{R}\ (i=1,2), &  \\
h_i: [0,T]\times \mathbb{R} \times \mathbb{R}\times \mathbb{R}^d \times
\mathbb{R}\times \mathbb{R}\times \mathbb{R}^d \times U_1 \times U_2
\rightarrow \mathbb{R},\ \ (i=1,2). &
\end{array}%
\end{eqnarray*}

\noindent Each player, having the same goal $\Phi(X_T)$, aims at minimizing
her/his cost functional $J_i(v_1(\cdot),v_2(\cdot))$ by selecting an
appropriate admissible control $v_i(\cdot)\in \mathcal{U}_i \ (i=1,2)$. The
problem is to find a Nash equilibrium point $(u_1(\cdot),u_2(\cdot)) \in
\mathcal{U}_1 \times \mathcal{U}_2$ for the non-zero sum game, such that
\begin{eqnarray}
\left\{
\begin{array}{ll}
J_1(u_1(\cdot),u_2(\cdot))=\min \limits_{v_1(\cdot) \in \mathcal{U}_1}
J_1(v_1(\cdot),u_2(\cdot)), &  \\
J_2(u_1(\cdot),u_2(\cdot))=\min \limits_{v_2(\cdot) \in \mathcal{U}_2}
J_2(u_1(\cdot),v_2(\cdot)). &
\end{array}
\right.
\end{eqnarray}
For simplicity, we denote the problem above by \itshape Problem \upshape(%
\itshape CFBNM\upshape).

In order to give the maximum principle, we assume that the following
hypothesis holds.
\begin{equation*}
\left\{
\begin{array}{ll}
\text{(i)}\ \ \ b,\sigma,f,\Phi,h_i,g_i\ \text{and}\ \gamma_i\
\mbox{are continuously
differentiable}; &  \\
\text{(ii)}\ \ \mbox{The derivatives of}\ b, \sigma,f \ \mbox{and} \ \Phi\ %
\mbox{are bounded}; &  \\
\text{(iii)}\ \mbox{The derivatives of}\ h_i\ \mbox{are bounded
by} \ C(1+|\tilde{x}|+|\tilde{y}|+|\tilde{z}|+|x|+|y|+|z|); &  \\
\text{(iv)}\ \mbox{The derivatives of}\ g_i\ \mbox{and} \gamma_i \
\mbox{with
respect to}\ x \ \mbox{and}\ y\ \mbox{are bounded by}\ C(1+|x|)\  &  \\
\ \ \ \ \ \ \mbox{and}\ C(1+|y|)\ \mbox{respectively}; &  \\
\text{(v)}\ \ \mbox{For any given pair of control}\ (v_1(\cdot),v_2(\cdot)),
\
\mbox{equation (\ref{state equation for
coupled game}) satisfies (H4) and (H5)}. &
\end{array}
\right. \leqno(\text{\textbf{A6}})
\end{equation*}

Let $(u_1(\cdot),u_2(\cdot))$ be a Nash equilibrium point of Problem (CFBNM)
and let $(X(\cdot),Y(\cdot),Z(\cdot))$ be the corresponding trajectory of
game system. In this fully coupled case, the adjoint equation, different
from the case in Section 6, has the form: for $i=1,2$,
\begin{eqnarray}
\left\{
\begin{array}{ll}
-dp_i(t)=\mathbb{E}^{\prime }\Big(\bar{b}_{\tilde{x}}(t)(p_i(t))^{\prime }+%
\bar{b}_{x}(t)p_i(t)+\bar{\sigma}_{\tilde{x}}(t)(q_i(t))^{\prime }+\bar{%
\sigma}_{x}(t)q_i(t)\Big)dt &  \\
\ \ \ \ \ \ \ \ \ \ \ \ \ \ +\mathbb{E}^{\prime }\Big(\bar{h}_{i\tilde{x}%
}(t) +\bar{h}_{ix}(t)-\bar{f}_{\tilde{x}}(t)(Q_i(t))^{\prime }-\bar{f}%
_{x}(t)Q_i(t)\Big)dt-q_i(t)dW_t, &  \\
dQ_i(t)=\mathbb{E}^{\prime }\Big(\bar{f}_{\tilde{y}}(t)(Q_i(t))^{\prime }+%
\bar{f}_{y}(t)Q_i(t)-\bar{b}_{\tilde{y}}(t)(p_i(t))^{\prime }-\bar{b}%
_{y}(t)p_i(t)-\bar{\sigma}_{\tilde{y}}(t)(q_i(t))^{\prime } &  \\
\ \ \ \ \ \ \ \ \ \ \ \ \ \ \ -\bar{\sigma}_{y}(t)q_i(t)-\bar{h}_{i\tilde{y}%
}(t)-\bar{h}_{iy}(t)\Big)dt+\mathbb{E}^{\prime }\Big(\bar{f}_{\tilde{z}%
}(t)(Q_i(t))^{\prime }+\bar{f}_{z}(t)Q_i(t)-\bar{b}_{z}(t)p_i(t) &  \\
\ \ \ \ \ \ \ \ \ \ \ \ \ \ \ -\bar{b}_{\tilde{z}}(t)(p_i(t))^{\prime }-\bar{%
\sigma}_{\tilde{z}}(t)(q_i(t))^{\prime }-\bar{\sigma}_{z}(t)q_i(t)-\bar{h}_{i%
\tilde{z}}(t) -\bar{h}_{iz}(t)\Big)dW_t, &  \\
p_i(T)=g_{ix}(X_T)-\Phi_{x}(X_T)Q_i(T),\ Q_i(0)=-\gamma_{iy}(Y(0)), \label%
{adjoint equation for coupled game} &
\end{array}
\right.
\end{eqnarray}
with $\bar{\psi}=\psi(t,X^{\prime }_t,Y^{\prime }_t,Z^{\prime
}_t,X_t,Y_t,Z_t,u_1(t),u_2(t))$, for $\psi=b,\sigma,f,h_1,h_2.$

This is a linear fully coupled mean-field FBSDE with bounded coefficients
under assumption (A6). It is easy to know that adjoint equation (\ref%
{adjoint equation for coupled game}) satisfies (H4) and (H6) since condition
(A6) and equation (\ref{state equation for coupled control}) satisfying (H4)
and (H5). From Theorem 3.2, this equation has a unique $\mathbb{F}$-adapted
solution $(p_i(\cdot),q_i(\cdot),Q_i(\cdot))$ such that
\begin{equation*}
\mathbb{E}\Big[\sup\limits_{0\leq t\leq T}|Q_i(t)|^2+\sup\limits_{0\leq
t\leq T}|p_i(t)|^2+\int_0^T|q_i(t)|^2dt\Big]<+\infty,\ \ \ \ \ \ \ \ \ \
(i=1,2).
\end{equation*}

\noindent We define the Hamiltonian function $H_i:[0,T]\times \mathbb{R}%
\times \mathbb{R}\times \mathbb{R}^d\times \mathbb{R}\times \mathbb{R}\times
\mathbb{R}^d \times \mathbb{R}\times \mathbb{R}^d \times \mathbb{R}\times
\mathbb{R}^k\times \mathbb{R}^k \rightarrow \mathbb{R}$ as
\begin{eqnarray}
\lefteqn{H_i(t,\tilde{x},\tilde{y},\tilde{z},x,y,z,p_i,q_i,Q_i,v_1,v_2)}
\notag \\
&=&p_i(t)b(t,\tilde{x},\tilde{y},\tilde{z},x,y,z,v_1,v_2)+q_i(t)\sigma(t,%
\tilde{x},\tilde{y},\tilde{z},x,y,z,v_1,v_2)
\label{Hamiltonian
for coupled game} \\
&&-f(t,\tilde{x},\tilde{y},\tilde{z},x,y,z,v_1,v_2)Q_i(t) +h_i(t,\tilde{x},%
\tilde{y},\tilde{z},x,y,z,v_1,v_2), \ \ \ (i=1,2).  \notag
\end{eqnarray}
The proof of the maximum principle and verification theorem in this case is
practically similar to Section 5. Thus we present these theorems without
proof.

\begin{theorem}
{\upshape(Stochastic Maximum Principle for SDGs of coupled FBSDEs)} Let (A6)
hold. If $(u_{1}(\cdot ),u_{2}(\cdot ))$ is a Nash equilibrium point of
Problem (CFBNM) and $(X(\cdot ),Y(\cdot ),Z(\cdot ))$ denotes the
corresponding trajectory, then for any $(v_{1},v_{2})\in U_{1}\times U_{2},$
the following maximum principle
\begin{eqnarray*}
&&\mathbb{E}\int_{0}^{T}\mathbb{E}^{\prime }\Big[H_{1v_{1}}\big(%
t,X_{t}^{\prime },Y_{t}^{\prime },Z_{t}^{\prime
},X_{t},Y_{t},Z_{t},p_{1}(t),q_{1}(t),Q_{1}(t),u_{1}(t),u_{2}(t)\big)%
(v_{1}(t)-u_{1}(t))\Big]dt\geq 0, \\
&&\mathbb{E}\int_{0}^{T}\mathbb{E}^{\prime }\Big[H_{2v_{2}}\big(%
t,X_{t}^{\prime },Y_{t}^{\prime },Z_{t}^{\prime
},X_{t},Y_{t},Z_{t},p_{2}(t),q_{2}(t),Q_{2}(t),u_{1}(t),u_{2}(t)\big)%
(v_{2}(t)-u_{2}(t))\Big]dt\geq 0,
\end{eqnarray*}%
hold a.s. a.e., where $(p_{i}(\cdot ),q_{i}(\cdot ),Q_{i}(\cdot ))\ (i=1,2)$
is the the solution of the adjoint equation (\ref{adjoint equation for
coupled game}) and the Hamiltonian function $H_{i}\ (i=1,2)$ is defined by (%
\ref{Hamiltonian for coupled game}).
\end{theorem}

\begin{theorem}
{\upshape (Sufficient Conditions for the Problem (CFBNM))} Assume that the
condition (A6) is satisfied. Let $(u_1(\cdot),u_2(\cdot))\in \mathcal{U}%
_1\times \mathcal{U}_2$ and $(X_t,Y_t,Z_t)$ be the corresponding state
trajectory. Suppose $(p_i(\cdot),q_i(\cdot),Q_i(\cdot)) \ (i=1,2)$ is the
solution of linear mean-field FBSDE (\ref{adjoint equation for coupled game}%
). Moreover, we assume functions $\Phi, g_i \ (i=1,2)$ are convex in $x$, $%
\gamma_i \ (i=1,2)$ is convex in $y$ and function $H_i(t,\tilde{x},\tilde{y},%
\tilde{z},x,y,z,p_i,q_i,Q_i,v_1,v_2) \ (i=1,2)$ is convex with respect to $(%
\tilde{x},\tilde{y},\tilde{z},x,y,z,v_1,v_2)$. Then, if
\begin{eqnarray*}
\lefteqn{\mathbb{E}%
'[H_1(t,X'_t,Y'_t,Z'_t,X_t,Y_t,Z_t,p_1(t),q_1(t),Q_1(t),u_1(t),u_2(t))]} \\
&&=\min\limits_{v_1 \in U_1}\mathbb{E}^{\prime }[H_1(t,X^{\prime
}_t,Y^{\prime }_t,Z^{\prime }_t,X_t,Y_t,Z_t,p_1(t),q_1(t),Q_1(t),v_1,u_2(t))]
\\
\lefteqn{\mathbb{E}%
'[H_2(t,X'_t,Y'_t,Z'_t,X_t,Y_t,Z_t,p_2(t),q_2(t),Q_2(t),u_1(t),u_2(t))]} \\
&&=\min\limits_{v_2 \in U_2}\mathbb{E}^{\prime }[H_2(t,X^{\prime
}_t,Y^{\prime }_t,Z^{\prime }_t,X_t,Y_t,Z_t,p_2(t),q_2(t),Q_2(t),u_1(t),v_2)]
\end{eqnarray*}
hold for all $t\in [0,T]$, $(u_1(\cdot),u_2(\cdot))$ is an equilibrium point
of Problem (CFBNM).
\end{theorem}

\section{Applications: Linear-Quadratic Case}

In this section, we give two LQ examples to illustrate our theoretical
results.

\begin{example}
For notational simplicity, we consider the following one-dimensional
stochastic control problem. Our aim is to search for the admissible control $%
u(\cdot)$ minimizing
\begin{equation}
J(v(\cdot))=\frac{1}{2}\mathbb{E}\Big[\int_0^Tv^2(t)dt+X_T^2+Y_0^2\Big],
\label{example 1 cost functional}
\end{equation}
subject to the following FBSDE:
\begin{eqnarray}
&&\ \ dX_t=\Big[\tilde{A}(t)\mathbb{E}[X_t]+A(t)X_t+B(t)v(t)\Big]dt+\Big[%
\tilde{C}(t)\mathbb{E}[X_t]+C(t)X_t+D(t)v(t)\Big]dW_t,  \notag \\
&&-dY_t=\Big[\tilde{a}(t)\mathbb{E}[X_t]+a(t)X_t+\tilde{b}(t)\mathbb{E}%
[Y_t]+b(t)Y_t+\tilde{\beta}(t)\mathbb{E}[Z_t]+\beta(t)Z_t+E(t)v(t)\Big]%
dt-Z_tdW_t,  \notag \\
&&X_0=a,\ \ Y_T=X_T,\ \ \ \ \ t\in [0,T],  \label{LQ example 1 state}
\end{eqnarray}
where $\tilde{A}(\cdot),\ A(\cdot),\ B(\cdot),\ \tilde{C}(\cdot), \
C(\cdot), D(\cdot),\ \tilde{a}(\cdot),\ a(\cdot),\ \tilde{b}(\cdot),\
b(\cdot), \tilde{\beta}(\cdot), \beta(\cdot)$ and $E(\cdot)$ are bounded and
deterministic, and $v(t), 0\leq t\leq T$ takes value in $\mathbb{R}$.

In this process, the Hamiltonian function is in the form of
\begin{eqnarray}
H(t,\tilde{x},\tilde{y},\tilde{z},x,y,z,p,q,Q,v)&= & p\Big[\tilde{A}(t)%
\tilde{x}+A(t)x+B(t)v\Big]+q\Big[\tilde{C}(t)\tilde{x}+C(t)x_t+D(t)v\Big]
\notag \\
&&-Q\Big[\tilde{a}(t)\tilde{x}+a(t)x+\tilde{b}(t)\tilde{y}+b(t)y+\tilde{\beta%
}(t)\tilde{z}+\beta(t)z+E(t)v\Big]  \notag \\
&&+\frac{1}{2}v^2,  \label{Hamiltonian example 1}
\end{eqnarray}
where $(p(\cdot),q(\cdot),Q(\cdot))$ satisfies
\begin{eqnarray}
\left\{
\begin{array}{ll}
dQ_t=\Big(\tilde{b}(t)\mathbb{E}[Q_t]+b(t)Q_t\Big)dt+\Big(\tilde{\beta}(t)%
\mathbb{E}[Q_t]+\beta(t)Q_t\Big)dW_t \nonumber &  \\
-dp_t=\Big(\tilde{A}(t)\mathbb{E}[p_t]+A(t)p_t+\tilde{C}(t)\mathbb{E}%
[q_t]+C(t)q_t-\tilde{a}(t)\mathbb{E}[Q_t]-a(t)Q_t\Big)dt-q_tdW_t, \nonumber
&  \\
Q_0=-Y_0,\ \ P_T=X_T-Q_T. &
\end{array}
\right.
\end{eqnarray}
If $u(\cdot)$ is optimal, then it follows from Theorem 5.1 and (\ref%
{Hamiltonian example 1}) that
\begin{equation}
u(t)=Q_tE(t)-p_tB(t)-q_tD(t), \ \ t\in [0,T].  \label{example 1
control}
\end{equation}
Moreover, it is easy to check that candidate optimal control (\ref{example 1
control}) is really the optimal control since the coefficients of Eq (\ref%
{LQ example 1 state}) and cost functional (\ref{example 1 cost functional})
satisfy the assumptions of Theorem 4.2.
\end{example}

\begin{example}
Let us consider the following forward-backward stochastic control system:

\begin{eqnarray}
&&\ \ dX_{t}=\Big[\tilde{b}(t)\mathbb{E}[X_{t}]+b(t)X_{t}+\tilde{A}(t)%
\mathbb{E}[Y_{t}]+A(t)Y_{t}+\tilde{B}(t)\mathbb{E}[Z_{t}]+B(t)Z_{t}+D(t)v(t)%
\Big]dt  \notag \\
&&\ \ \ \ \ \ \ \ \ \ \ +\Big[\tilde{\beta}(t)\mathbb{E}[X_{t}]+\beta
(t)X_{t}-\tilde{B}(t)\mathbb{E}[Y_{t}]-B(t)Y_{t}+\tilde{C}(t)\mathbb{E}%
[Z_{t}]+C(t)Z_{t}+E(t)v(t)\Big]dW_{t},  \notag \\
&&-dY_{t}=\Big[\tilde{a}(t)\mathbb{E}[X_{t}]+a(t)X_{t}+\tilde{b}(t)\mathbb{E}%
[Y_{t}]+b(t)Y_{t}+\tilde{\beta}(t)\mathbb{E}[Z_{t}]+\beta (t)Z_{t}+G(t)v(t)%
\Big]dt-Z_{t}dW_{t},  \notag \\
&&X_{0}=a,\ \ Y_{T}=RX_{T},\ \ \ \ \ t\in \lbrack 0,T],
\label{LQ fully coupled
state}
\end{eqnarray}%
where $R>0$ is a constant and $v\in L_{\mathbb{F}}^{2}(0,T;U)$. For
simplicity we also suppose that $U=\mathbb{R}$. Functions $\tilde{a}(\cdot
)>0,\ a(\cdot )>0,\tilde{A}(\cdot )<0,\ \ \tilde{C}(\cdot )<0,\ A(\cdot
)<0,C(\cdot )<0,\ \tilde{B}(\cdot ),\ \tilde{b}(\cdot ),\ \tilde{\beta}%
(\cdot ),\ B(\cdot ),\ D(\cdot ),\ E(\cdot )$, $G(\cdot ),\ b(\cdot )$ and $%
\beta (\cdot )$ are bounded and deterministic. For any given $v(\cdot )$, it
is easy to show that condition (H4) and monotonic condition (H5) hold. Then
from Theorem 7, the fully coupled Mean-field FBSDEs (\ref{LQ fully coupled
state}) has a unique solution $(X(\cdot ),Y(\cdot ),Z(\cdot ))$.

The cost functional is
\begin{equation}
J(v(\cdot ))=\frac{1}{2}\mathbb{E}\int_{0}^{T}\Big[L(t)v^{2}(t)\Big]dt+%
\mathbb{E}[MX_{T}^{2}+NY_{0}^{2}],
\end{equation}%
where constants $M>0,\ N>0$. Function $L(\cdot )$ is deterministic and
bounded, and $L^{-1}$ is also bounded. By (\ref{Hamiltonian for coupled
control}), the Hamiltonian function is given by
\begin{eqnarray*}
H(t,\tilde{x},\tilde{y},\tilde{z},x,y,z,p,q,Q,v) &=&p\Big[\tilde{b}(t)\tilde{%
x}+b(t)x+\tilde{A}(t)\tilde{y}+A(t)y+\tilde{B}(t)\tilde{z}+B(t)z+D(t)v\Big]
\\
&&+q\Big[\tilde{\beta}(t)\tilde{x}+\beta (t)x-\tilde{B}(t)\tilde{y}-B(t)y+%
\tilde{C}(t)\tilde{y}+C(t)z+E(t)v\Big] \\
&&-Q\Big[\tilde{a}(t)\tilde{x}+a(t)x+\tilde{b}(t)\tilde{y}+b(t)y+\tilde{\beta%
}(t)\tilde{z}+\beta (t)z+G(t)v\Big] \\
&&+\frac{1}{2}L(t)v^{2}.
\end{eqnarray*}%
According to Theorem 17, if $u(\cdot )$ is optimal, then
\begin{equation}
u(t)=-L^{-1}(t)\big(p_{t}D(t)+q_{t}E(t)-Q_{t}G(t)\big),\ \ \ \ \ 0\leq t\leq
T,  \label{example 2 control}
\end{equation}%
where $(p(\cdot ),q(\cdot ),Q(\cdot ))$ is the solution of the following
fully coupled Mean-field FBSDEs
\begin{equation*}
\left\{
\begin{array}{ll}
dQ_{t}=\Big(\tilde{b}(t)\mathbb{E}[Q_{t}]+b(t)Q_{t}-\tilde{A}(t)\mathbb{E}%
[p_{t}]-A(t)p_{t}+\tilde{B}(t)\mathbb{E}[q_{t}]+B(t)q_{t}\Big)dt &  \\
\ \ \ \ \ \ \ \ \ \ +\Big(\tilde{\beta}(t)\mathbb{E}[Q_{t}]+\beta (t)Q_{t}-%
\tilde{B}(t)\mathbb{E}[p_{t}]-B(t)p_{t}-\tilde{C}(t)\mathbb{E}%
[q_{t}]-C(t)q_{t}\Big)dW_{t} &  \\
-dp_{t}=\Big(\tilde{b}(t)\mathbb{E}[p_{t}]+b(t)p_{t}+\tilde{\beta}(t)\mathbb{%
E}[q_{t}]+\beta (t)q_{t}-\tilde{a}(t)\mathbb{E}[Q_{t}]-a(t)Q_{t}\Big)%
dt-q_{t}dW_{t},\label{example 2 adjoint equation} &  \\
Q_{0}=-2NY_{0},\ \ p_{T}=2M_{T}X_{T}-RQ_{T},\ \ \ \ \ t\in \lbrack 0,T]. &
\end{array}%
\right.
\end{equation*}%
Similarly, it is easy to verify that the monotonic condition (H6) holds,
then from Theorem 10, FBSDEs (\ref{example 2 adjoint equation}) admits a
unique solution $(Q(\cdot ),p(\cdot ),q(\cdot ))$.

Moreover, since $g(x)=M_Tx^2,\ \gamma(y)=Ny^2,\ \Phi(x)=Rx$ are convex and $%
H(t,\tilde{x},\tilde{y},\tilde{z},x,y,z,p,q,Q,v)$ is convex in $(\tilde{x},%
\tilde{y},\tilde{z},x,y,z,v)$, we can know that the admissible control (\ref%
{example 2 control}) which satisfying the necessary condition of optimality
is really an optimal control.
\end{example}


\begin{thebibliography}{99}
\bibitem{Hu Peng} Y. Hu, S. Peng: Solution of forward-backward stochastic
differential equations, Probab. Theory Relat. Fields 103 (1995) 273-283.

\bibitem{Peng 1991} S. Peng: Probabilistic interpretation for systems of
quasilinear parabolic partial differential equations, Stochastics, 37
(1991), 61-74.

\bibitem{Duffine Epstein} D. Duffie, L. Epstein: Asset pricing with
stochastic differential utilities, Rev. Financial Stud, 5 (1992), 411-436.

\bibitem{Peng Wu} S. Peng, Z. Wu: Fully coupled forward-backward stochastic
differential equations and applications to the optimal control, SIAM J.
Control Optim. 37(3), (1999) 825-843

\bibitem{Shi Wu 2010} J. Shi, Z. Wu: Maximum principle for
partially-observed Optimal control of fully-coupled forward-backward
stochastic systems, J Optim Theory Appl (2010) 145: 543-578

\bibitem{Antonelli} F. Antonelli: Backward-forward stochastic differential
equations, Ann. Appl. Probab. 3, (1993) 777-793 .

\bibitem{Four step} J. Ma, P. Protter, J. Yong: Solving forward-backward
stochastic differential equations explicitly-a four step scheme. Probab.
Theory Relat. Fields 98, (1994) 339-359

\bibitem{Yong} J. Yong.: Finding adapted solutions of forward-backward
stochastic differential equations: method of continuation. Probab. Theory
Relat. Fields 107, (1997) 537-572

\bibitem{Pardoux Tang} E. Pardoux, S. Tang: Forward-backward stochastic
differential equations and quasilinear parabolic PDEs. Probab. Theory Relat.
Fields 114, (1999) 123-150

\bibitem{Delarue} F. Delarue: On the Existence and Uniqueness of Solutions
to FBSDEs in a Non-Degenerate Case. Stoch. Process. Appl. 99, (2002) 209-286

\bibitem{Zhang} J. Zhang: The wellposedness of FBSDEs. Discrete Contin. Dyn.
Syst., Ser. B 6, (2006) 927-940

\bibitem{MFBSDE1} R. Buckdahn, B. Djehiche, J. Li, S. Peng: Mean-field
backward stochastic differential equations. A limit approach, Ann. Probab.
37 (4) (2009) 1524-1565.

\bibitem{MFBSDE2} R. Buckdahn, J. Li, S. Peng: Mean-field backward
stochastic differential equations and related patial differential equations,
Stoch. Process. Appl. 119 (10) (2009) 3133-3154 .

\bibitem{Lions} J.M. Lasry, P.L. Lions: Mean field games,Japan. J. Math. 2
(2007) 229-260.

\bibitem{Djehiche} D. Andersson, B. Djehiche: A Maximum Principle for SDEs
of Mean-Field Type, Appl Math Optim. 63 (2011) 341-356.

\bibitem{Juan Li General SMP} R. Buckdahn, B. Djehiche, J. Li: A General
Stochastic Maximum Principle for SDEs of Mean-Field Type, Appl Math Optim.
(2011).

\bibitem{Zhou} T. Meyer-Brandis, B. $\emptyset $sendal, X.Y. Zhou: A
mean-field stochastic maximum principle via Malliavin calculus. (A Special
issue for Mark Davis, Festschrift) (2010).

\bibitem{Juan Li Automatica} J. Li: Stochastic maximum principle in the
mean-field controls, Automatica 48(2012) 366-373.

\bibitem{Bensoussan} A. Bensoussan: Lectures on stochastic control. In:
Mitter, S.K., Moro, A. (eds.) Nonlinear Filtering and Stochastic Control.
Springer Lecture Notes in Mathematics, vol. 972. Springer, Berlin (1982)

\bibitem{Peng 1993} S. Peng: Backward stochastic differential equations and
application to optimal control. Applied Mathematics and Optimization, 27(4)
(1993) 125-144

\bibitem{Shi Wu 2006} J. Shi, Z. Wu: The maximum principle for fully coupled
forward-backward stochastic control system. Acta Autom Sin, 32: (2006)
161-169

\bibitem{Wang Yu} G. Wang, Z. Yu: A Pontryagin's Maximum Principle for
Non-Zero Sum Differential Games of BSDEs with Applications. IEEE
Transactions on Automatic control, 55 (7), (2010) 1742-1747.

\bibitem{Isaacs 1965} R. Isaacs: Differential Games, Wiley, New York, 1965.

\bibitem{Fleming Souganidis 1989} W. H. Fleming and P. E. Souganidis: On the
existence of value functions of two-player, zero-sum stochastic differential
games, Indiana Univ. Math. J., 38 (1989), 293-314.

\bibitem{Friedman 1971} A. Friedman: Differential Games, Wiley, New York,
1971.

\bibitem{Evans Souganidis 1984} L. C. Evans and P. E. Souganidis:
Differential games and representation formulas for solutions of
Hamilton-Jacobi-Isaacs equations, Indiana Univ. Math. J., 33 (1984), 773-797.

\bibitem{Hamadene 1999} S. Hamad$\acute{e}$ne: Nonzero-sum linear-quadratic
stochastic differential games and backward-forward equations, Stochastic
Anal. Appl., vol. 17, (1999) 117-130.

\bibitem{Lim Zhou 2001} A. E. B. Lim and X. Zhou: Risk-sensitive control
with HARA utility,\ IEEE Trans. Autom. Control, vol. 46, no. 4, (2001)
563-578.

\bibitem{Altman 2005} E. Altman: Applications of dynamic games in queues, in
Advances in Dynamic Games. Boston, MA: Birkhauser, vol. 7, (2005) 309-342.

\bibitem{Ikeda Watanabe} N. Ikeda, S. Watanabe: Stochastic differential
equations and diffusion processes. Amsterdam-Tokyo: North Holland-Kodansha.
(1989)

\bibitem{Karatzas} I. Karatzas, S. E. Shreve: Brownian motion and stochastic
calculus. Springer, (1987)
\end{thebibliography}
\end{document}